\documentclass[11pt]{article}

\usepackage[margin=1in]{geometry}
\usepackage[utf8]{inputenc}

\usepackage{amsmath,amsthm,amssymb,enumerate,multirow, booktabs, color, scrextend}
\usepackage[dvipsnames]{xcolor}
\usepackage{tikz}
\usepackage{pgfplots, pgfplotstable}
\pgfplotsset{compat = newest}

\usepackage[normalem]{ulem}
\usepackage{comment}

\usepackage[caption=false]{subfig}
\usepackage{graphicx,epstopdf}

\usepackage{footnote}
\makesavenoteenv{algorithm}
\makesavenoteenv{align}
\makesavenoteenv{tabular}
\makesavenoteenv{table}

\usepackage{algorithm}
\usepackage{algorithmic}

\usepackage[margin=1in]{geometry}
\usepackage[fleqn,tbtags]{mathtools}
\usepackage[english]{babel}

\usepackage{cleveref}

\title{Convergence analysis of a norm minimization-based convex vector optimization algorithm\thanks{This work was funded by T\"UB{\.I}TAK (Scientific \& Technological
		Research Council of Turkey), Project No. 118M479.}}
\author{\c{C}a\u{g}{\i}n Ararat\thanks{Bilkent University, Department of Industrial Engineering, Ankara, Turkey, cararat@bilkent.edu.tr.}
	\and
	Firdevs Ulus\thanks{Bilkent University, Department of Industrial Engineering, Ankara, Turkey, firdevs@bilkent.edu.tr.}
	\and
	Muhammad Umer\thanks{National University of Sciences and Technology (NUST), Department of Industrial Engineering, Islamabad, Pakistan, mhmdumer@cae.nust.edu.pk.}
}

\allowdisplaybreaks

\addto\captionsenglish {}
\makeatletter \renewenvironment{proof}[1][\proofname] {\par\pushQED{\qed}\normalfont\topsep6\p@\@plus6\p@\relax\trivlist\item[\hskip\labelsep\bfseries#1\@addpunct{.}]\ignorespaces}{\popQED\endtrivlist\@endpefalse} \makeatother
\DeclareMathOperator*{\argmax}{arg\,max}

\newcommand{\N}{\mathbb{N}}
\newcommand{\R}{\mathbb{R}}

\newcommand{\out}{\text{out}}
\newcommand{\inn}{\text{in}}

\DeclareMathOperator{\wMin}{wMin}
\DeclareMathOperator{\Min}{Min}

\DeclareMathOperator{\cl}{cl}
\DeclareMathOperator{\bd}{bd}
\DeclareMathOperator{\Int}{int}
\DeclareMathOperator{\conv}{conv}
\DeclareMathOperator{\cone}{cone}

\DeclareMathOperator{\ext}{ext}
\DeclareMathOperator{\recc}{recc}

\newcommand{\dom}{{\rm dom\,}}

\newcommand{\norm}[1]{\lVert#1\rVert}
\newcommand{\abs}[1]{\lvert#1\rvert}
\newtheorem{theorem}{Theorem}
\numberwithin{equation}{section}
\numberwithin{theorem}{section}
\newtheorem{proposition}[theorem]{Proposition}
\newtheorem{lemma}[theorem]{Lemma}
\newtheorem{corollary}[theorem]{Corollary}
\newtheorem{definition}[theorem]{Definition}
\newtheorem{remark}[theorem]{Remark}

\newtheorem{assumption}[theorem]{Assumption}
\newtheorem{notation}[theorem]{Notation}

\newtheorem{problem}[theorem]{Problem}

\newcounter{algo}
\newcommand{\of}[1]{\ensuremath{\left( #1 \right)}}
\newcommand{\cb}[1]{\ensuremath{ \left\{ #1 \right\} }}
\newcommand{\sqb}[1]{\ensuremath{ \left[ #1 \right] }}

\begin{document}
\maketitle
\thispagestyle{empty}

\begin{abstract}
In this work, we propose an outer approximation algorithm for solving bounded convex vector optimization problems (CVOPs). The scalarization model solved iteratively within the algorithm is a modification of the norm-minimizing scalarization proposed in \cite{ararat2021norm}. For a predetermined tolerance $\epsilon>0$, we prove that the algorithm terminates after finitely many iterations, and it returns a polyhedral outer approximation to the upper image of the CVOP such that the Hausdorff distance between the two is less than $\epsilon$. We show that for an arbitrary norm used in the scalarization models, the approximation error after $k$ iterations decreases by the order of $\mathcal{O}(k^{{1}/{(1-q)}})$, where $q$ is the dimension of the objective space. An improved convergence rate of $\mathcal{O}(k^{{2}/{(1-q)}})$ is proved for the special case of using the Euclidean norm.\\
\\[-5pt]
\\[-5pt]
\end{abstract}

\section{Introduction}

Vector optimization refers to minimizing a vector-valued objective function with respect to a given order relation over a feasible region. The special case where this order relation is induced by the positive orthant yields multiobjective optimization. We refer the reader \cite{armand, evans1973revised, newton, psimplex, econometrica} for multiobjective optimization algorithms working in the decision space and to \cite{benson1998outer,ehrgott_shao2008,tradeoffs,bbeichfelder} for algorithms working in the objective space.  

In this paper, we consider vector optimization problems in which the order relation is induced by a polyhedral ordering cone. 
Having efficient solution approaches for 
vector optimization problems is crucial in many application areas including financial mathematics \cite{ararat2019computation,FeiRud2017}, economics \cite{RudUlu2021}, optimization and control \cite{RudKov2021, shao2019numerical}, and machine learning \cite{ararat2023vector,giesen2019using,giesen2019efficient}. Additionally, 
vector optimization algorithms have been used in solving other challenging mathematical problems such as convex projections \cite{KovacovaRudloff2021}, time inconsistent dynamic optimization problems \cite{RudKov2021}, nonzero-sum games \cite{feinstein2023characterizing} and mixed integer multiobjective optimization problems \cite{DeSanEicNieRoc2020,bbeichfelder}. 

Specifically, we focus on convex vector optimization problems (CVOPs) in which the objective function is cone-convex and the feasible region is a convex set. There are different solution concepts and algorithms for CVOPs in the literature. Many solution approaches are objective space-based, and they generate polyhedral approximations to the set of all (weakly) minimal elements in the objective space. To that end, they consider the \emph{upper image}, which is the image of the feasible region plus the ordering cone, and iteratively generate polyhedral outer approximations to it by solving scalarization problems. See, for instance, \cite{hamel2014benson,lohne2011vector} for algorithms to solve linear vector optimization problems, and \cite{ararat2023geometric, ararat2021norm, vertexselection, ehrgott2011approximation, keskin_ulus21, klamroth2007constrained, lohne2014primal} for algorithms to solve bounded CVOPs. Recently, in \cite{unboundedCVOP_23}, algorithms and solution concepts are proposed for unbounded CVOPs.

In this paper, we focus on bounded CVOPs, propose an objective space-based outer approximation algorithm, and study its convergence rate. Note that the literature on the convergence analysis of such algorithms is very limited. For some special cases, the finiteness has been studied without determining the convergence rate. This is the case for some linear vector optimization algorithms, see \cite{lohne2011vector}. Moreover, an outer approximation algorithm for nonlinear multiobjective optimization problems (MOPs), in which the ordering cone is the positive orthant, is proposed and proved to be finite in \cite{nobakhtian2017benson}. Recently, Ararat et al. \cite{ararat2021norm} proposed an outer approximation algorithm for general CVOPs and showed its finiteness. Their result is mainly based on constructing polyhedral outer approximations of the upper image such that the vertices are guaranteed to be within a compact set in each iteration.

On the other hand, in 2003, Klamroth et al. \cite{klamroth2003unbiased} proposed approximation algorithms for convex and non-convex MOPs, and the convergence rates of the algorithms are provided for the biobjective case. In particular, they proved that, for biobjective problems, the approximation error after $k$ iterations of the algorithms in \cite{klamroth2003unbiased} decreases by the order of $\mathcal{O}(k^{-2})$. Later, in 2007, Klamroth and Tind \cite{klamroth2007constrained} provided the convergence rate of similar algorithms for the multiobjective setting. Accordingly, if there are $q \geq 3$ objective functions, then the approximation error after $k$ iterations decreases by the order of $\mathcal{O}(k^{{2}/{(1-q)}})$. It is also noted that this convergence rate is in general the best possible rate, see also \cite{gruber1992aspects}. Note that the algorithms in \cite{klamroth2007constrained,klamroth2003unbiased} cannot be directly applied to CVOPs as they require the ordering cone to be the positive orthant, which is not the case for general CVOPs. 

The convergence rate provided in \cite{klamroth2007constrained} is based on existing results on the convergence rate of algorithms for finding polyhedral approximations of convex compact sets. Indeed, there is rich literature on this subject; see \cite{kamenev1992class, kamenev2002conjugate, kamenev2012polyhedral, lotov2013interactive}, 
for instance. In general, the results for convex compact sets cannot be directly used to establish the convergence rate of algorithms for CVOPs since the upper image is an unbounded set by its structure. However, the design of the MOP algorithms in \cite{klamroth2007constrained}, which depends heavily on the symmetry of the orthants, allowed them to use the results in \cite{kamenev1992class, lotov2013interactive} directly to establish the convergence rate. 

In this paper, we propose an outer approximation algorithm for CVOPs and study its convergence rate. To the best of our knowledge, this is the first paper providing a convergence analysis of a general CVOP algorithm in the literature. The proposed algorithm is based on a norm-minimizing scalarization, which is similar to the one proposed in \cite{ararat2021norm}. We modify this scalarization by adding a single constraint which enables us to work with a compact subset of the upper image, which still contains the set of all (weakly) minimal elements on its boundary. 

Note that any norm on the objective space can be used to form the scalarization models solved within the algorithm. We prove that, for any norm, the proposed algorithm is finite, and the approximation error after $k$ iterations decreases by the order of $\mathcal{O}(k^{{1}/{(1-q)}})$, where $q$ is the dimension of the objective space. In particular, we prove that the sequence of outer approximations formed by the algorithm is an \emph{H-sequence of cutting}, see \Cref{def:LotovH}, and then apply the results in \cite{kamenev1992class, lotov2013interactive}, directly. 

Moreover, we consider the special case of using the Euclidean norm within the scalarizations separately and prove that the approximation error decreases by the order of $\mathcal{O}(k^{2/(1-q)})$. Different from the previous case, the results in \cite{kamenev1992class, lotov2013interactive} are not directly applicable to establish the improved convergence rate. Hence, the convergence analysis for the outer approximating polytopes produced by the proposed algorithm is proved from scratch.

The paper is organized as follows: \Cref{sec:Prelim} provides basic definitions, notations and the problem description, and it recalls some well-known results used throughout the paper. In \Cref{sec:scal}, the norm-minimizing scalarization is introduced and corresponding results are presented. The algorithm is explained in \Cref{sec:alg2} and its finiteness is proved in \Cref{sec:fin}. The convergence analysis of the algorithm is presented for any norm and for the special case of Euclidean norm in Sections \ref{sec:conv} and \ref{sec:conv_l2}, respectively. Finally, we provide computational results in \Cref{sec:experiments} and conclude the paper in \Cref{sec:concl}.

\section{Preliminaries and problem definition} \label{sec:Prelim}

Throughout the paper, let $\R^q$ denote the $q$-dimensional Euclidean space, where $q\in\N:=\{1,2,\ldots\}$, and $\norm{\cdot}$ be an arbitrary fixed norm on $\R^q$ with dual norm $\norm{\cdot}_\ast$. For $y\in\R^q$ and $\epsilon>0$, we define $\mathbb{B}_{\epsilon}(y) \coloneqq \{z\in\R^q\mid \norm{y-z}\leq \epsilon\}$ as the closed ball centered at $y$ with radius $\epsilon$.

Let $A \subseteq \mathbb{R}^q$ be a set. We write $\Int A$, $\cl A$, $\bd A$, $\conv A$, $\cone A$ for the interior, closure, boundary, convex hull, conic hull of $A$, respectively. The \emph{indicator function} $I_A$ of $A$ is defined by $I_A(z)=0$ for $z\in A$ and by $I_A(z)=+\infty$ for $z\in\R^q\setminus A$. When $A,B\subseteq\R^q$ are nonempty sets and $\lambda\in\R$, we define $A+B:=\{y+z\mid y\in A, z\in B\}$, $\lambda A:=\{\lambda y\mid y\in A\}$, and $A-B:=A+(-1)B$. When $A$ is a nonempty convex set, the set of all $k\in\R^q$ such that $A+\cone\{k\}\subseteq A$ is the \emph{recession cone} of $A$, denoted by $\recc A$; a nonzero element of $\recc A$ is a \emph{recession direction} of $A$. When $A$ is a nonempty convex cone, the \emph{dual cone} of $A$, given by $A^+ := \{w \in \R^q \mid \forall y \in A : w^\mathsf{T}y \geq 0 \}$, is a closed convex cone.

A cone $C \subseteq \R^q$ is said to be \emph{solid} if $\Int C\neq\emptyset$ and \emph{pointed} if $C \cap -C = \{0\}$. Let $C\neq\R^q$ be a pointed convex cone. The relation $\leq_C$ on $\R^q$ defined by $y \leq_C z$ if and only if $z - y \in C$, $y,z\in\R^q$, is a partial order.

Let $A \subseteq \R^q$ be a set and let $y\in A$. We say that $y$ is a \emph{$C$-minimal element} of $A$ if there is no $z\in A\setminus\{y\}$ with $z\leq_C y$. When $C$ is solid, we say that $y$ is a \emph{weakly $C$-minimal element} of $A$ if $(\{y\} -\Int C) \cap A = \emptyset$. Let $\Min_C A$ denote the set of all $C$-minimal elements of $A$, and $\wMin_C A$ denote the set of all weakly $C$-minimal elements of $A$ whenever $C$ is solid.

Let $A, B\subseteq\R^q$ be nonempty sets. The Hausdorff distance between $A$ and $B$ is defined as
\[
\delta^H(A,B) \coloneqq \max\cb{\sup_{y \in A}d(y,B),\ \sup_{z\in B}d(z,A)},
\]
where $d(z,A):=\inf_{y\in A}\norm{z-y}$. By \cite[Proposition 3.2]{hausdorffsurvey}, we also have
\[
\delta^H(A, B) =\inf \cb{\epsilon > 0 \mid A \subseteq  B + \mathbb{B}_{\epsilon}(0),\ B \subseteq A + \mathbb{B}_{\epsilon}(0)}.
\]

Let $A\subseteq\R^q$ be a convex set and $w\in\R^q\setminus\{0\}$. We introduce the halfspace
\begin{equation}\label{eq:H}
	\mathcal{H}(w,A)\coloneqq \cb{z\in\R^q\mid w^{\mathsf{T}}z \geq \inf_{z^\prime\in A}w^{\mathsf{T}}z^\prime}.
\end{equation}
We also write $\mathcal{H}(w,y)\coloneqq \mathcal{H}(w,\{y\})$ for every $y\in\R^q$.	If $y\in A$ is such that $w^{\mathsf{T}}y=\inf_{z^\prime\in A}w^{\mathsf{T}}z^\prime$, then we have $\mathcal{H}(w,A)=\mathcal{H}(w,y)$. In this case, $\bd \mathcal{H}(w,A)=\{z \in \R^q\mid w^{\mathsf{T}}z = w^{\mathsf{T}}y\}$ is called a \emph{supporting hyperplane} of $A$ at $y$ and $\mathcal{H}(w,A)\supseteq A$ is called a \emph{supporting halfspace} of $A$ at $y$. We say that a point $v\in A$ is an \emph{extreme point} of $A$ if there does not exist $y,z\in A$ and $0<\lambda <1$ such that $y\neq z$ and $v = \lambda y +(1-\lambda)z$. We denote the set of all extreme points of $A$ by $\ext A$. If $A$ is polyhedral, then $\ext A$ coincides with the set of all vertices of $A$. 

The following lemmas are simple observations that will be used throughout.
\begin{lemma} \label{lem:ShiftedH}
	Let $w\in \R^q$ with $\norm{w}_\ast \leq 1$ and $y\in \R^q$. Then, $\mathcal{H}(w, y)+ \mathbb{B}_{\frac{\epsilon}{2}}(0) \subseteq \mathcal{H}(w, y,\epsilon)$, where 
	\begin{equation}\label{eq:shiftedH}
		\mathcal{H}(w, y,\epsilon)\coloneqq \cb{z \in \R^q \mid w^\mathsf{T}z \geq w^\mathsf{T}y - \frac{\epsilon}{2}}.
	\end{equation} 
\end{lemma}
\begin{proof}
	Let $y^\prime  \in \mathcal{H}({w}, y)$ and $y{^{\prime\prime}}\in \mathbb{B}_{\frac{\epsilon}{2}}(0)$. Then, using $y^\prime  \in \mathcal{H}({w}, y)$, H\"older's inequality and the facts that $\norm{w}_\ast \leq 1$ and  $\norm{y{^{\prime\prime}}} \leq \frac{\epsilon}{2}$, we have  \begin{align}\notag
		w^{\mathsf{T}} (y^\prime + y{^{\prime\prime}}) \geq  {w}^{\mathsf{T}} y + w^{\mathsf{T}} y{^{\prime\prime}} \geq {w}^{\mathsf{T}} y -\norm{w}_\ast \norm{y{^{\prime\prime}}}  \geq {w}^{\mathsf{T}} y - \frac{\epsilon}{2},
	\end{align}
	which completes the proof.
\end{proof}

\begin{lemma} \label{lem:PconeP}
	Let $A\subseteq \R^q$ be a nonempty convex compact set and $C\subseteq \R^q$ be a nonempty convex cone. Then, $\ext (A+C) \subseteq \ext A$.
\end{lemma}

\begin{proof}
	The inclusion is trivial if $\ext(A+C) = \emptyset$. Otherwise, let $v \in \ext(A+C) \setminus \ext(A)$ for a contradiction. Since $v \in A + C$, we have $v = a + c$ for some $a \in A$ and $c \in C$. Note that if $c=0$, then $v=a \notin \ext A$ implies that $v$ can be written as a non-trivial convex combination of points from $A \subseteq A+C$. As this contradicts $v \in \ext(A+C)$, we conclude that $c \neq 0$. Next, 
	for every $\lambda \geq 0$, we have $a + \lambda c \in A + C$.	For $\lambda > 1$, we can write $v$ as a strict convex combination of $a, a+\lambda c \in A+C$ as $v = a + c = \frac{\lambda - 1}{\lambda} a + \frac1\lambda (a + \lambda c)$. This is a contradiction to $v \in \ext(A+C)$.
\end{proof}

The mathematical problem that is the subject of this study is provided next.

\begin{problem} \label{problem}
	We consider a \emph{convex vector optimization problem (CVOP)} given by
	\begin{align}
		\text{minimize } \Gamma(x) \text{ with respect to } \leq_C \text{ subject to } x \in \mathcal{X}, \tag{P}\label{eq:P}
	\end{align}
	where $C \subsetneq\R^q$ is a closed convex solid and pointed cone, $n\in \N$, $X\subseteq \mathbb{R}^n$ is a convex set, $\Gamma\colon X \to \mathbb{R}^q$ is a $C$-convex and continuous vector-valued function, and $\mathcal{X}\subseteq X$ is a compact convex set with $\Int \mathcal{X}\neq \emptyset$. 
\end{problem}

Recall that $\Gamma$ is called \emph{$C$-convex} if $\Gamma(\lambda x_1 + (1-\lambda)x_2) \leq_C \lambda \Gamma(x_1) + (1 - \lambda) \Gamma(x_2)$ for every $x_1, x_2 \in X, \lambda \in [0,1]$; equivalently, for every $w\in C^+$, the function $x\mapsto w^{\mathsf{T}}\Gamma(x)$ on $X$ is convex.

The set $\mathcal{P} := \Gamma(\mathcal{X}) + C$ is called the \emph{upper image} of \eqref{eq:P}. Under the assumptions of \Cref{problem}, $\mathcal{P}$ is a closed convex set with $\mathcal{P}=\mathcal{P}+C$, see \cite[Remark~3.2]{ararat2021norm}. Moreover, \eqref{eq:P} is a \emph{bounded} CVOP in the sense that there exists a point $y \in \mathbb{R}^q$ such that $\mathcal{P} \subseteq \{y\} + C$, see \cite[Definition 3.1]{lohne2014primal}.

The next definition recalls the relevant solution concepts for the CVOP given by \Cref{problem}.

\begin{definition}\label{defn:finite epsilon-solution}
	Consider \Cref{problem}.
	\begin{enumerate} 
		\item A point $\bar{x} \in \mathcal{X}$ is called a \emph{(weak) minimizer} for \eqref{eq:P} if $\Gamma(\bar{x})\in \Min_C \Gamma(\mathcal{X})$ ($\Gamma(\bar{x})\in \wMin_C \Gamma(\mathcal{X})$). 
		\item 
		\cite[Definition 3.3]{vertexselection} Let $\bar{\mathcal{X}} \subseteq \mathcal{X}$ be a nonempty finite set of (weak) minimizers and $\mathcal{\bar{P}} :=\conv \Gamma(\bar{\mathcal{X}}) + C $. Then, $\bar{\mathcal{X}}$ is called a \emph{finite (weak) $\epsilon$-solution} of \eqref{eq:P} if
		$\mathcal{\bar{P}} + \mathbb{B}_{\epsilon}(0)  \supseteq \mathcal{P}$.
	\end{enumerate}
\end{definition}

\begin{remark}\label{rem:Hausdorff}
	$\mathcal{\bar{P}} + \mathbb{B}_{\epsilon}(0)  \supseteq \mathcal{P}$ is equivalent to $\delta^H(\mathcal{P}, \mathcal{\bar{P}}) \leq \epsilon$, see \cite[Remark 3.6]{ararat2021norm}.
\end{remark}

Below, we list some relevant results from the literature regarding \Cref{problem}.

\begin{proposition}\label{prop:solution exists}\cite[Proposition 3.8]{ararat2021norm}
	For every $\epsilon > 0$, there exists a finite $\epsilon$-solution of \eqref{eq:P}.
\end{proposition}

For each $w \in C^+\setminus\{0\}$, let us consider the corresponding \emph{weighted sum scalarization} of \eqref{eq:P}:
\begin{equation}\label{eq:P1(w)}
	\text{minimize } w^\mathsf{T}\Gamma(x) \text{ subject to }x \in \mathcal{X} .  \tag{WS$(w)$} 
\end{equation}

\begin{proposition}\label{prop:jahn2009vector} \cite[Corollary 2.3]{jahn1984}
	Let $w \in C^+ \setminus \{0\}$. Then, every optimal solution of \eqref{eq:P1(w)} is a weak minimizer of \eqref{eq:P}.
\end{proposition}

\section{A norm-minimizing scalarization} \label{sec:scal}

In this section, we discuss a modification of the following norm-minimizing scalarization model, which is introduced in \cite{ararat2021norm} for CVOPs:
\begin{equation}\label{eq:P(v)}
	\text{minimize} \ \norm{z} \ \text{subject to} \ \Gamma(x) - z - v \leq_C 0, \ x \in \mathcal{X},\ z \in \R^q, \tag{NM$(v)$} 
\end{equation}
where $v\in\R^q$. Let us also provide the Lagrange dual of \eqref{eq:P(v)} below, see \cite[Section 4]{ararat2021norm} for details:
\begin{equation}\label{eq:D(v)}
	{\text{maximize} \inf_{x \in \mathcal{X}} w^{\mathsf{T}}\Gamma(x) - w^{\mathsf{T}}v \ \text{subject to} \ } \norm{w}_* \leq 1, w \in C^+. \tag{{dNM}$(v)$} 
\end{equation}

We modify \eqref{eq:P(v)} by adding another constraint as follows:
\begin{equation}\label{eq:P2(v)}
	\text{minimize} \ \norm{z} \ \text{subject to} \ \Gamma(x) - z - v \leq_C 0, \ \bar{w}^{\mathsf{T}}(v + z) \leq \gamma , \ x \in \mathcal{X},\ z \in \R^q, 
	\tag{P$(v)$}
\end{equation}
where $\bar{w} \in \Int C^+$ is assumed to be fixed, and $\gamma\in\R$ is a parameter. Note that this is a convex program as both the objective function and the feasible region are convex.

\begin{remark}\label{rem:dist}
	It is not difficult to observe that the original norm-minimizing scalarization \eqref{eq:P(v)} computes the distance $d(v,\mathcal{P})$ from the point $v$ to the upper image $\mathcal{P}$. On the other hand, the optimal value of \eqref{eq:P2(v)} may not be equal to $d(v,\mathcal{P})$. Indeed, it finds the distance $d(v, \mathcal{P} \cap S(\gamma))$, where
	\begin{align}\label{eq:Sdefn}
		S(\gamma) := \{ y \in \R^q \mid \bar{w}^{\mathsf{T}} y \leq \gamma\}.
	\end{align}
	Clearly, if $\gamma$ is not sufficiently large, then $\mathcal{P}\cap S(\gamma)$ may be the empty set, hence \eqref{eq:P2(v)} may become infeasible. The construction of a suitable choice of $\gamma$ is discussed in \cite[Section 6]{ararat2021norm} and it will be recalled later in \Cref{sec:alg2}.
\end{remark}

Note that \eqref{eq:P2(v)} can be equivalently written as 
\begin{equation*} \label{eq:Pprime}
	\text{minimize} \ f(x, z) \ \text{subject to} \ G(x, z) \cap -(C \times \R_+) \neq \emptyset, \  (x, z) \in X \times \mathbb{R}^q,      
	\tag{P$'({v})$}
\end{equation*}
where the scalar function $f\colon X \times \R^q \to \overline{\R}$ and the set-valued function $G : X \times \mathbb{R}^q \rightrightarrows \mathbb{R}^{q+1}$ are given by
\begin{align}\label{eq:fGdefn}
	f(x, z) := \norm{z} + I_{\mathcal{X}}(x), \: G(x, z) := \{ (\Gamma(x) - z - v, \bar{w}^{\mathsf{T}}(v + z) -{\gamma}) \}, \: x \in X, z \in \R^q.
\end{align}

Using the equivalent problem \eqref{eq:Pprime} and following similar steps for constructing the Lagrange dual \eqref{eq:D(v)} of \eqref{eq:P(v)} as in \cite[Section 4]{ararat2021norm}, the Lagrange dual of \eqref{eq:P2(v)} can be formulated as
\begin{equation}\label{eq:D2(v)} 
	\text{maximize~} \ \phi(w, \lambda) \ \text{~subject to~} \ w \in C^+, \ \lambda\geq 0, 
	\tag{D$(v)$} 
\end{equation} 
where the dual objective function $\phi\colon \R^{q+1}\to \overline{\R}$ is given by
\[
\phi(w, \lambda) := \inf_{x \in \mathcal{X}, z \in \R^q} \of{\norm{z} + w^{\mathsf{T}}(\Gamma(x) - z - v) + \lambda (\bar{w}^{\mathsf{T}}(v + z) - {\gamma})}.
\]
Then, the optimal value of \eqref{eq:D2(v)} can be written as
\begin{align}
	\sup_{(w, \lambda) \in C^+ \times \R_+} \phi(w, \lambda) = \sup \bigg\{ \inf_{x \in \mathcal{X}}  w^{\mathsf{T}}(\Gamma(x) - v) + \lambda (\bar{w}^{\mathsf{T}} v - {\gamma}) \mid  \norm{w - \lambda \bar{w}}_\ast \leq 1,\ (w, \lambda) \in C^+ \times \R_+\bigg\}, \notag
\end{align}
where we use the fact that the conjugate function of $\norm{\cdot}$ is the indicator function of the unit ball of the dual norm $\norm{\cdot}_\ast$; see \cite[Example 3.26]{boyd2004convex}. Next, we define $\tilde{w} := w - \lambda \bar{w}$ and rewrite the value of \eqref{eq:D2(v)} as 
\begin{align} \label{eq:pre_dualpr_eq}
	\sup \cb{ \inf_{x \in \mathcal{X}} {\tilde{w}^{\mathsf{T}}} (\Gamma(x) - v) + \lambda (\bar{w}^{\mathsf{T}} \Gamma(x) - {\gamma}) \mid \norm{\tilde{w}}_\ast \leq 1,\ \tilde{w} + \lambda \bar{w} \in C^+, \lambda \geq 0}.
\end{align}

We establish the strong duality between \eqref{eq:P2(v)} and \eqref{eq:D2(v)} next.

\begin{proposition}\label{prop:optsol}
	Let $\gamma\in\R$ be such that $\Gamma(\mathcal{X}) \subseteq \Int S(\gamma)$, where $S(\gamma)$ is given by \eqref{eq:Sdefn}.
	Then, for every $v \in \mathbb{R}^q$, there exist optimal solutions $(x^v, z^v)$ and $(w^v, \lambda^v)$ of problems \eqref{eq:P2(v)} and \eqref{eq:D2(v)}, respectively, and the optimal values coincide.  
\end{proposition}

\begin{proof}
	First, we show that there exists a feasible solution $(\tilde{x}, \tilde{z})$ of \eqref{eq:P2(v)}. 
	Fix $\tilde{x} \in \mathcal{X}$, which exists since $\Int\mathcal{X} \neq \emptyset$ is assumed. 
	Let us define $\tilde{z} := \Gamma(\tilde{x}) - v$. Then, $	v + \tilde{z} = \Gamma(\tilde{x}) \in \Int S(\gamma) \subseteq S(\gamma)$ and	$	v + \tilde{z} = \Gamma(\tilde{x}) \in \{\Gamma(\tilde{x})\} + C$, that is, $		\Gamma(\tilde{x}) \leq_C v + \tilde{z}$. Hence, $(\tilde{x}, \tilde{z})$ is feasible.
	
	Next, note that the feasible region	of \eqref{eq:P2(v)} is a subset of $\mathcal{X} \times (\mathcal{P} \cap S(\gamma) - \{v\}) \subseteq \R^{n + q}$. Moreover, $\mathcal{X}$ is compact by the problem definition and, using	\cite[Lemma 6.3]{ararat2021norm}, it is easy to see that $\mathcal{P} \cap S$ is compact as well. Hence, $\mathcal{X} \times (\mathcal{P} \cap S - \{v\})$ is a compact set, which implies that the feasible region of \eqref{eq:P2(v)} is bounded. To show that it is also closed, consider a sequence of feasible points $(x_n, z_n)_{n \in\N}$ such that $\lim_{n \rightarrow \infty}(x_n, z_n) = (x, z)$. We have $x \in \mathcal{X}$ as $\mathcal{X}$ is compact and $z \in S(\gamma) - v$ as $S(\gamma)$ is a closed halfspace. Since $\Gamma \colon X \rightarrow \R^q$ is continuous, $C$ is a closed cone and $v+z_n - \Gamma(x_n) \in C$ for every $n \in\N$, we obtain $\lim_{n \rightarrow \infty}(v + z_n - \Gamma(x_n)) = v + z - \Gamma(x) \in C$. Hence, an optimal solution $(x^*, z^*)$ of \eqref{eq:P2(v)} exists by Weierstrass theorem. 
	
	For strong duality, we show that the following constraint qualification in \cite{khan2016set, lohne2011vector} holds for \eqref{eq:P2(v)}:
	\begin{equation}\label{eq:CQset}
		G(\dom f) \cap -\Int (C \times \R_+) \neq \emptyset,
	\end{equation}
	where $f$ and $G$ are as in \eqref{eq:fGdefn}, and $\dom f\coloneqq\{(x,z)\in X\times\R^q\mid f(x,z)<+\infty\}=\mathcal{X}\times\R^q$. Since $\mathcal{X}$ has nonempty interior, there exists $x^0 \in$ int $\mathcal{X}$. Moreover, since $\Gamma(\mathcal{X}) \subseteq \Int S(\gamma)$, there exists $\epsilon > 0$ such that $B(\Gamma(x^0), \epsilon) \subseteq \Int S(\gamma)$. Let $c \in \Int C$ be fixed and define $y^0 \coloneqq \Gamma(x^0) + \frac{\epsilon}{2} \frac{c}{\norm{c}}$. Clearly, $y^0 \in B(\Gamma(x^0), \epsilon) \cap (\{\Gamma(x^0)\} + \Int C)$. Defining $z^0 \coloneqq y^0 - v$, we obtain $v + z^0 - \Gamma(x^0) \in \Int C$. Moreover, since $y^0 \in B(\Gamma(x^0), \epsilon) \subseteq \Int S(\gamma)$, we have
	$\bar{w}^{\mathsf{T}} (v + z^0) < \gamma$. Then,
	\begin{equation*}
		G(x^0, z^0) = \{ (\Gamma(x^0) - z^0 - v, \bar{w}^{\mathsf{T}} (v + z^0) -\gamma ) \} \subseteq - \Int (C \times \R_+). 
	\end{equation*}
	On the other hand, since $\dom f = \mathcal{X} \times \R^q$, we have $G(x^0, z^0) \subseteq G(\dom f)$. Hence, \eqref{eq:CQset} is satisfied.
	
	Recall that $\mathcal{X}\subseteq \R^n$ is convex and $\Gamma\colon X \to \R^q$ is a $C$-convex function. Then, by standard arguments, it can be shown that $f : X \times \R^q \rightarrow \mathbb{\overline{R}}$ is convex and $G \colon X \times \R^q \rightrightarrows \R^{q+1}$ is $(C \times \R_+)$-convex {as a set-valued function}. Together with \eqref{eq:CQset}, these imply strong duality and dual attainment by \cite[Theorem~3.19]{lohne2011vector}.
\end{proof}

\begin{notation}
	From now on, by $(x^v, z^v)$ and $(w^v, \lambda^v)$, we denote arbitrary optimal solutions of problems \eqref{eq:P2(v)} and \eqref{eq:D2(v)}, respectively. Moreover, we define $\tilde{w}^v \coloneqq w^v - \lambda^v \bar{w}$.
\end{notation}

The next lemma relates an optimal solution $(w^v,\lambda^v)$ of \eqref{eq:D2(v)} with the possible regions for the parameter $v$. Although the result is analogous to \cite[Lemma 4.5]{ararat2021norm}, it works under the additional assumption that $v\in S(\lambda)$.

\begin{lemma}\label{lem:vlemma}
	Let $\gamma\in\R$ be such that $\Gamma(\mathcal{X}) \subseteq \Int S(\gamma)$. Let $v \in S(\gamma)$. Then, the following statements hold:
	\begin{enumerate}[(a)]
		\item If $v\notin\mathcal{P}$, then $z^v\neq 0$, $w^v\neq 0$, and $\tilde{w}^v \neq 0$. 
		\item If $v\in\bd\mathcal{P}$, then $z^v=0$.
		\item If $v\in\Int\mathcal{P}$, then $z^v=0$ and $w^v=0$. 
	\end{enumerate}
	In particular, $v\in\mathcal{P}$ if and only if $z^v=0$.
\end{lemma}

\begin{proof}
	By the feasibility of \eqref{eq:P2(v)} and strong duality, $z^v\neq 0$ and $w^v\neq 0$ in (a), see the proof of \cite[Lemma 4.5]{ararat2021norm} for details. For the other claim of (a), assume to the contrary that $w^v - \lambda^v \bar{w}=0$. From the optimal value of \eqref{eq:D2(v)} and strong duality, we get
	\begin{align}\notag
		\lambda^v \of{\inf_{x \in \mathcal{X}} \bar{w}^{\mathsf{T}}\Gamma(x) - \gamma} = \norm{z^v},
	\end{align}
	where $\lambda^v \geq 0$ and $\bar{w}^{\mathsf{T}} \Gamma(x) - \gamma \leq 0$ for any $x \in \mathcal{X}$, as $\Gamma(\mathcal{X}) \subseteq S(\gamma)$. Hence, $\norm{z^v} \leq 0$. This implies, $\norm{z^v} = 0$, a contradiction.
	
	To prove (b) and (c), note that we have $\bar{w}^\mathsf{T} v \leq \gamma$. Then, the assertion follows using the similar arguments as in the proof of \cite[Lemma 4.5]{ararat2021norm}.
\end{proof}

The following two propositions interpret the primal $(x^v, z^v)$ and the dual optimal solutions $(w^v, \lambda^v)$, that is, primal solution yields a weak minimizer for problem \eqref{eq:P} and dual solution provides a supporting hyperplane of $\mathcal{P} \cap S(\gamma)$ at $y^v = v +z^v$.

\begin{proposition}\label{prop:zv0}
	If $v \notin \Int \mathcal{P}$ and $v \in S(\gamma)$, 
	then $x^v$ is a weak minimizer of \eqref{eq:P}, and $y^v := v + z^v \in \wMin_C \mathcal{P}$.
\end{proposition}

\begin{proof}
	By using similar arguments as in the proof of \cite[Proposition 4.6]{ararat2021norm}, $x^v$ is a weak minimizer of \eqref{eq:P}. For the second claim, first note that $y^v \in \mathcal{P}$ since $(x^v,z^v)$ is feasible for \eqref{eq:P2(v)}. To get a contradiction, assume that $y^v \notin \wMin_C \mathcal{P}$; hence, $y^v = v + z^v  \in \Int \mathcal{P}$ with $\norm{z^v} \neq 0$ since $v \notin \Int \mathcal{P}$ by our assumption. Then, there exists
	$0 < \epsilon \leq \norm{z^v}$ such that
	\begin{align}\notag
		v + z^v - \epsilon\frac{z^v}{\norm{z^v}} \in \mathcal{P},
	\end{align}
	which implies the existence of $\bar{x} \in \mathcal{X}$ with
	\begin{align}\notag
		v + (\norm{z^v} - \epsilon) \frac{z^v}{\norm{z^v}} \in \{\Gamma(\bar{x})\} + C.
	\end{align}
	Let $\bar{z} := (\norm{z^v} - \epsilon) \frac{z^v}{\norm{z^v}}$. We show that $\bar{w}^\mathsf{T} (v + \bar{z}) \leq \gamma$. Assume otherwise that
	\begin{align}\label{eq:y_wMin}
		\bar{w}^\mathsf{T} (v + \bar{z}) =	\bar{w}^\mathsf{T} \of{v + z^v - \epsilon\frac{z^v}{\norm{z^v}}} > \gamma.
	\end{align}
	Since $\bar{w}^\mathsf{T} (v + z^v) \leq \gamma$ by the feasibility of \eqref{eq:P2(v)},
	\begin{align}\notag
		- \epsilon\frac{\bar{w}^\mathsf{T} z^v}{\norm{z^v}} > \gamma - \bar{w}^\mathsf{T} (v + z^v) \geq \gamma - \gamma = 0.
	\end{align}
	This implies $\bar{w}^\mathsf{T} z^v < 0$. We also know that $\bar{w}^\mathsf{T} v \leq \gamma$ as $v \in S(\gamma)$. Then, from \eqref{eq:y_wMin}, we get
	\begin{align}\notag
		\bar{w}^\mathsf{T} z^v \of{1 - \frac{\epsilon}{\norm{z^v}}} > \gamma - \bar{w}^\mathsf{T} v \geq \gamma - \gamma = 0.
	\end{align}
	This implies, $1 - \frac{\epsilon}{\norm{z^v}} < 0$, that is, $\norm{z^v} < \epsilon$. This contradicts with the upper bound of $\epsilon$.
	
	Now, $(\bar{x}, \bar{z})$ is feasible for \eqref{eq:P2(v)}. But this gives $\norm{\bar{z}} < \norm{z^v}$, which is a contradiction to the optimality of $(x^v,z^v)$. 
\end{proof}

\begin{proposition}\label{prop:supp_halfspace}
	Suppose that $\tilde{w}^v \neq 0$. Let 
	$v \in S(\gamma)$ and $\Gamma(\mathcal{X}) \subseteq \Int S(\gamma)$. Then,
	\begin{align}\notag
		\mathcal{H}(\tilde{w}^v, y^v) = \{y \in \mathbb{R}^q \mid (\tilde{w}^v)^{\mathsf{T}}y \geq (\tilde{w}^v)^{\mathsf{T}} y^v \}
	\end{align}
	is a supporting halfspace of $\mathcal{P}\cap S(\gamma)$ at $y^v=v+z^v$.	
\end{proposition}

\begin{proof}
	Let us fix $y \in \mathcal{P} \cap S(\gamma)$, consider the problems \eqref{eq:P2(v)} and \eqref{eq:D2(v)}. Then,
	\begin{align}\notag
		0 \geq \inf_{x \in \mathcal{X}} (w^v)^{\mathsf{T}}(\Gamma(x) - y) + \lambda^v (\bar{w}^{\mathsf{T}} y - \gamma);
	\end{align}
	see the proof of \cite[Proposition 4.7]{ararat2021norm}. This gives
	\begin{align}\label{eq:4}
		(w^v)^{\mathsf{T}}y - \lambda^v \bar{w}^{\mathsf{T}} y \geq \inf_{x \in \mathcal{X}} (w^v)^{\mathsf{T}}\Gamma(x) - \lambda^v \gamma. 
	\end{align}
	Using the strong duality between \eqref{eq:P2(v)} and \eqref{eq:D2(v)}, we get
	\begin{align}\label{eq:5}
		\norm{z^v}  &= \inf_{x \in \mathcal{X}} (w^v)^{\mathsf{T}}(\Gamma(x) - v) + \lambda^v (\bar{w}^{\mathsf{T}} v - \gamma) \\ \notag
		&= \inf_{x \in \mathcal{X}} (w^v)^{\mathsf{T}}\Gamma(x) - \lambda^v \gamma - (w^v)^{\mathsf{T}} v + \lambda^v \bar{w}^{\mathsf{T}} v.
	\end{align}
	From \eqref{eq:4} and \eqref{eq:5}, we obtain
	\begin{align}\notag
		(w^v)^{\mathsf{T}}y - \lambda^v \bar{w}^{\mathsf{T}} y \geq \norm{z^v} + (w^v)^{\mathsf{T}}v - \lambda^v \bar{w}^{\mathsf{T}} v,
	\end{align}
	that is,
	\begin{align}\notag
		(\tilde{w}^v)^{\mathsf{T}} y \geq \norm{z^v} + (\tilde{w}^v)^{\mathsf{T}} v \geq (\tilde{w}^v)^\mathsf{T}z^v + (\tilde{w}^v)^{\mathsf{T}} v = (\tilde{w}^v)^{\mathsf{T}} y^v,
	\end{align}
	where the second inequality is a consequence of the definition of dual norm and the dual constraint $\norm{\tilde{w}^v}_* \leq 1$. Hence, $y\in \mathcal{H}(\tilde{w}^v, y^v)$. Since $y$ is arbitrary, this implies $\mathcal{P} \cap S(\gamma) \subseteq \mathcal{H}(\tilde{w}^v, y^v)$.
	
	Next, by \Cref{prop:zv0}, we have $y^v \in \mathcal{P}$; in particular, $y^v \in \mathcal{P} \cap S(\gamma)$ by the feasibility of \eqref{eq:P2(v)}. We have $y^v \in \bd \mathcal{H}(\tilde{w}^v, y^v)$ by the definition of $\mathcal{H}(\tilde{w}^v, y^v) $. Hence, $y^v \in \mathcal{P} \cap S(\gamma) \cap \bd \mathcal{H}(\tilde{w}^v, y^v)$, which completes the proof. 
\end{proof}

\section{The algorithm} \label{sec:alg2}

In this section, we propose an algorithm for solving \Cref{problem} whose finiteness will be proved in \Cref{sec:fin}, followed by its convergence rate in \Cref{sec:conv}. \Cref{alg1} is an outer approximation algorithm which has a similar structure as Algorithm 2 in \cite{ararat2021norm}. It consists of two phases. The first one is an initialization phase which computes a compact initial outer approximation $\bar{\mathcal{P}}^{\out}_0$ using the original norm-minimizing scalarization \eqref{eq:P(v)} and its dual \eqref{eq:D(v)}. In the second phase, the current outer approximation is refined in each iteration via a supporting halfspace that is calculated by solving the modified norm-minimizing scalarization \eqref{eq:P2(v)} and its dual \eqref{eq:D2(v)}. The algorithm works under the following assumption.

\begin{assumption} \label{assmp:C_poly}
	The ordering cone $C$ is polyhedral. 
\end{assumption}

\Cref{assmp:C_poly} ensures that the dual cone $C^+$ is also polyhedral. 
Letting $w^1, \ldots,$ $w^J$ be the generating vectors of $C^+$ with $J \in \N$, we have $C^+ = \conv \cone \{w^1,\ldots,w^J\}$. Note that $C^+$ is solid as $C$ is assumed to be pointed, see \cite[Propostion 2.4.3]{facchinei2007finite}. In particular, $J \geq q$. Without loss of generality, we assume that $\|w^j\|_\ast = 1$ for each $j\in \{1,\ldots, J\}$. 

The initialization phase starts by solving the scalarizations (WS$(w^j)$), $j\in \{1,\ldots,J\}$: an optimal solution $x^j$ of (WS$(w^j)$) exists by the assumptions of \Cref{problem}, and $x^j$ is a weak minimizer of \eqref{eq:P} by \Cref{prop:jahn2009vector}. This gives the set ${\mathcal{X}}_0 \coloneqq \{x^1,\ldots,x^J\}$ of weak minimizers. We define 
\begin{equation}\label{eq:P_0}
	\mathcal{P}^{\out}_0 \coloneqq \bigcap_{j=1}^J \{y \in \mathbb{R}^q \mid (w^j)^{\mathsf{T}}y \geq (w^j)^{\mathsf{T}} \Gamma(x^j) \}.
\end{equation}
Note that $\mathcal{P}^{\out}_0 \supseteq \mathcal{P}$, see \cite[Section 5]{ararat2021norm}. Then, $\mathcal{P}^{\out}_0$ is further intersected with the halfspace $S(\gamma)=\{ y\in\R^q \mid \bar{w}^{\mathsf{T}}y \leq \gamma\}$ (see \eqref{eq:Sdefn}), where $\bar{w}\in \Int C^+$, $\gamma\in\R$ are fixed such that
\begin{equation}\label{eq:gamma}
	\bar{w} \coloneqq \frac{\sum_{j = 1}^{J} w^j}{\norm{\sum_{j = 1}^{J} w^j}_\ast},\quad \gamma > \sup_{x \in \mathcal{X}} \bar{w}^{\mathsf{T}} \Gamma(x) + \max_{v \in \ext\mathcal{P}_0^{\out}} (\bar{w}^{\mathsf{T}} v - \beta)^+ + \delta^H(\mathcal{P}^{\out}_0, \mathcal{P}).
\end{equation}
Here, $\beta \geq \sup_{x\in\mathcal{X}}\bar{w}^{\mathsf{T}} \Gamma(x)$ is a constant and $a^+\coloneqq\max\{a,0\}$ for $a\in\R$. The calculation of $\gamma$ is discussed in \cite[Section 6]{ararat2021norm} in detail. \eqref{eq:gamma} ensures that $\Gamma(\mathcal{X})\subseteq \mathcal{P}\cap S(\gamma)$. Finally, we obtain the initial outer approximation of $\mathcal{P}\cap S(\gamma)$ as $\bar{\mathcal{P}}^\out_0 := \mathcal{P}^\out_0 \cap S(\gamma)$, which is compact by \cite[Lemma 6.3]{ararat2021norm}. 

Next, in the second phase of the algorithm, the set $\ext \bar{\mathcal{P}}^\out_k$ of all vertices of $\bar{\mathcal{P}}^\out_k$ is computed, $k\geq 0$. Then, for each $v \in \ext \bar{\mathcal{P}}^\out_k$, optimal solutions $(x^v, z^v)$ and $(w^v, \lambda^v)$ are obtained by solving the modified norm-minimizing scalarization \eqref{eq:P2(v)} and its dual \eqref{eq:D2(v)}, respectively. If $\|z^v\| \leq \epsilon$, where $\epsilon>0$ is a predetermined approximation error, then the set $\mathcal{X}_k$ of weak minimizers is updated with $x^v$, see \Cref{prop:zv0}. To find one of the farthest vertices to the set $\mathcal{P} \cap S(\gamma)$, we compute
\begin{equation}\label{eq:zmax}
	v^k \in \argmax \{\norm{z^v} \mid v \in \ext \bar{\mathcal{P}}^\out_k\}.
\end{equation}
If $\|{z^{v^{k}}}\| > \epsilon$, then the supporting halfspace of $\mathcal{P}$ at $y^{v^{k}}$, namely,
\begin{equation}\label{eq:Hk}
	\mathcal{H}_k\coloneqq \mathcal{H}(\tilde{w}^{v^{k}}, y^{v^{k}}) =\{y \in \mathbb{R}^q \mid (\tilde{w}^{v^{k}})^{\mathsf{T}}y \geq (\tilde{w}^{v^{k}})^{\mathsf{T}} y^{v^{k}} \}
\end{equation}
used to update the current outer approximation as $\bar{\mathcal{P}}^\out_{k+1} \coloneqq \bar{\mathcal{P}}^\out_{k} \cap \mathcal{H}_k$; see \Cref{prop:supp_halfspace}. Otherwise, the algorithm terminates.

\setcounter{algo}{1}
\begin{algorithm}
	\caption{Outer Approximation Algorithm for \eqref{eq:P}}
	\label{alg1}
	\algsetup{
		linenosize=\small,
		linenodelimiter=.
	}
	\begin{algorithmic}[1]
		\STATE Compute an optimal solution $x^j$ of (WS$(w^j)$) for each $j \in \{1,\ldots,J\}$;
		\STATE $k = 0, \bar{\mathcal{{X}}}_0 = \{x^1,\dots,x^J\}, \mathcal{V}^{\text{known}} = \emptyset, \mathcal{V}^{\text{known2}} = \emptyset$;
		\STATE Store an $H$-representation of $\mathcal{P}^{\out}_0$ according to \eqref{eq:P_0};
		\STATE Compute $\ext \mathcal{P}^\out_0$ from the $H$-representation of $\mathcal{P}^\out_0$;
		\STATE Compute $\gamma$ by \eqref{eq:gamma} using \cite[Remark 6.1]{ararat2021norm};
		\STATE $\bar{\mathcal{P}}^\out_{0} = \mathcal{P}^\out_0 \cap S(\gamma)$;
		\REPEAT
		\STATE Stop $\gets \TRUE$;
		\STATE Compute $\ext \bar{\mathcal{P}}^\out_k$ from the $H$-representation of $\bar{\mathcal{P}}^\out_k$;
		\FOR{$v \in \ext \bar{\mathcal{P}}^\out_k \setminus \mathcal{V}^{\text{known}}$} 
		\IF {$v \notin \mathcal{V}^{\text{known2}}$}
		\STATE Solve \eqref{eq:P2(v)} and \eqref{eq:D2(v)} to compute $(x^v,z^v)$ and $(w^v, \lambda^v)$;
		\ENDIF
		\IF{$\norm{z^v} \leq \epsilon$}
		\STATE $\bar{\mathcal{X}}_{k} \gets \bar{\mathcal{X}}_{k}\cup \{x^v\}$, $\mathcal{V}^{\text{known}} \gets \mathcal{V}^{\text{known}} \cup \{v\}$;
		\ELSE
		\STATE $\mathcal{V}^{\text{known2}} \gets \mathcal{V}^{\text{known2}} \cup \{v\}$;
		\ENDIF
		\ENDFOR
		\STATE Compute $v^k \in  \argmax \{\norm{z^v} \mid v \in \ext \bar{\mathcal{P}}^\out_k\}$;
		\IF{$\|{z^{v^{k}}}\| > \epsilon$}
		\STATE $\bar{\mathcal{P}}^\out_{k+1} = \bar{\mathcal{P}}^\out_k \cap \mathcal{H}_k$, $\bar{\mathcal{X}}_{k+1} = \bar{\mathcal{X}}_k$;
		\STATE $k \leftarrow k + 1$;
		\STATE Stop $\gets \FALSE$;
		\ENDIF
		\UNTIL{Stop}
		\RETURN		 
		$ \begin{cases} 
			\bar{\mathcal{{X}}}_k &: \text{A finite weak $\epsilon$-solution to \eqref{eq:P}}; \\
			\bar{\mathcal{P}}^\out_k &: \text{An outer approximation of~} \mathcal{P} \cap S.\\
		\end{cases}$
	\end{algorithmic}
\end{algorithm}

\begin{remark}\label{rem:Scontains} 
	By the definition of $S(\gamma)$, we have $\Gamma(\mathcal{X})\subseteq S(\gamma)$. Since $\bar{\mathcal{P}}^\out_0 \supseteq \mathcal{P} \cap S(\gamma)$ and $\bar{\mathcal{P}}^\out_{k+1} = \bar{\mathcal{P}}^\out_k\cap\mathcal{H}_k$, we have $\bar{\mathcal{P}}^\out_k \supseteq \mathcal{P} \cap S(\gamma) \supseteq \Gamma(\mathcal{X})$ for every $k\geq 0$. Then, using \cite[Remark 3.2]{ararat2021norm}, we obtain	$\mathcal{P} = \Gamma(\mathcal{X}) + C \subseteq \bar{\mathcal{P}}^\out_k + C$. Hence, $\bar{\mathcal{P}}^\out_k + C$ gives an outer approximation of $\mathcal{P}$. 
\end{remark}

\begin{remark}\label{rem:Pequal}
	We have $	\mathcal{P} = (\mathcal{P} \cap S(\gamma)) + C.$ Indeed, $\mathcal{P} = \Gamma(\mathcal{X}) + C \subseteq (\mathcal{P} \cap S(\gamma)) + C$ as $\Gamma(\mathcal{X})\subseteq S(\gamma)$. The other inclusion is by $(\mathcal{P} \cap S(\gamma)) + C \subseteq \mathcal{P} + C = \mathcal{P}$.
\end{remark}

\begin{lemma}\label{lem:vintP}
	For every $k\geq 0$ and $v\in \ext \bar{\mathcal{P}}^\out_k$, we have $v \notin \Int \mathcal{P}$.
\end{lemma}

\begin{proof}
	Clearly, $\bd\mathcal{P}^\out_0 \subseteq \R^q \setminus \Int \mathcal{P} $ since $\mathcal{P}^\out_0 \supseteq \mathcal{P}$. Then, for every $v \in \ext\bar{\mathcal{P}}^\out_0 \subseteq \bd \mathcal{P}^\out_0$, we have $v \notin \Int \mathcal{P}$. 
	
	Assume that $k\geq 1$ is the first iteration number for which there exists $v\in \ext \bar{\mathcal{P}}^\out_k$ such that $v \in \Int \mathcal{P}$. Note that $\ext \bar{\mathcal{P}}^\out_k \subseteq S(\gamma)$ by the construction of the algorithm. Since $v$ is a vertex of $\bar{\mathcal{P}}^\out_k$, it must be true that $v\in\bd \mathcal{H}_{\bar{k}}$ for some $\bar{k}\leq k$. Moreover, $\bd \mathcal{H}_{\bar{k}}$ is a supporting hyperplane of the closed set $\mathcal{P} \cap S(\gamma)$. Hence,
	\begin{align}\notag 
		v \notin \Int (\mathcal{P} \cap S(\gamma)) = \Int \mathcal{P} \cap \Int S(\gamma).
	\end{align}
	Then, we must have $v\in \bd S(\gamma)$ since $\ext \bar{\mathcal{P}}^\out_k \subseteq S(\gamma)$. As $v \in \Int \mathcal{P}$, there exists $\delta > 0$ such that $\mathbb{B}_\delta(v) \subseteq \Int \mathcal{P}$. Clearly, $\mathbb{B}_\delta(v) \cap \bd S(\gamma) \neq \emptyset$. Let us fix $v' \in \bd S(\gamma)$, where $v' \neq v$, and define
	\begin{align}\notag 
		s \coloneqq \delta\frac{v - v'}{\norm{v - v'}}.
	\end{align}
	Then, $v \pm s$ is an affine combination of $v, v' \in \bd S(\gamma)$. Since $\bd S(\gamma)$ is an affine set, we get $v \pm s \in \bd S(\gamma)$. This implies that
	\begin{align}\notag 
		v \pm s \in \Int \mathcal{P} \cap \bd S(\gamma) \subseteq \mathcal{P} \cap S(\gamma) \subseteq \bar{\mathcal{P}}^\out_k.
	\end{align}
	This implies that $v$ can be written as a convex combination of $v \pm s \in \bar{\mathcal{P}}^\out_k$ since $v = \frac12 (v + s) + \frac12 (v - s)$. This is a contradiction to $v\in \ext\bar{\mathcal{P}}^\out_k$. Hence, $v \notin \Int \mathcal{P}$.
\end{proof}

\begin{theorem}\label{thm:alg2}
	Under \Cref{assmp:C_poly}, \Cref{alg1} works correctly: if the algorithm terminates, then it returns a finite weak $\epsilon$-solution to \eqref{eq:P}.
\end{theorem}

\begin{proof}
	Note that $\mathcal{\bar{X}}_0$ consists of weak minimizers and $\ext\bar{\mathcal{P}}^\out_0$ is a nonempty set of vertices, see the proof of \cite[Theorem 6.6]{ararat2021norm} for details. Let $k\geq 0$ and $v\in\ext\bar{\mathcal{P}}^\out_k$. Then, optimal solutions ($x^v$,$z^v$) and $(w^v, \lambda^v)$ to \eqref{eq:P2(v)} and \eqref{eq:D2(v)}, respectively, exist by \Cref{prop:optsol}. By \Cref{lem:vintP}, $v \notin \Int \mathcal{P}$. Hence, by \Cref{prop:zv0}, $x^v$ is a weak minimizer of \eqref{eq:P}.
	If $v^{k}\in \bd\mathcal{P}$, see line 20, then we have $z^{v^{k}} = 0$ by \Cref{lem:vlemma}. In this case, line 21 ensures that the outer approximation is not updated and the algorithm terminates. If $v^k \notin \mathcal{P}$, which implies $\tilde{w}^{v^{k}}\neq 0$ by \Cref{lem:vlemma}, then $\mathcal{H}_k$ given by \eqref{eq:Hk} is a supporting halfspace of $\mathcal{P} \cap S(\gamma)$ by \Cref{prop:supp_halfspace}. We have $\bar{\mathcal{P}}^\out_k \supseteq \mathcal{P} \cap S(\gamma)$ as $\bar{\mathcal{P}}^\out_0 \supseteq \mathcal{P} \cap S(\gamma)$; $\mathcal{H}_1, \ldots, \mathcal{H}_k$ are supporting halfspaces of $\mathcal{P}\cap S(\gamma)$; and $\bar{\mathcal{P}}^\out_{k} = \bar{\mathcal{P}}^\out_0\cap \mathcal{H}_1\cap\ldots\cap\mathcal{H}_k$. The latter also implies that $\ext\bar{\mathcal{P}}^\out_k\neq \emptyset$.
	
	Assume that the algorithm stops after $\hat{k}$ iterations. Consequently, $\bar{\mathcal{X}}_{\hat{k}}$ is a finite set of weak minimizers. Then, using similar arguments as in the proof of \cite[Theorem 6.6]{ararat2021norm}, one can show that \Cref{defn:finite epsilon-solution} holds, that is,
	\begin{align}\notag
		\conv \Gamma(\bar{\mathcal{X}}_{\hat{k}}) + C + \mathbb{B}_\epsilon(0) \supseteq \mathcal{P},
	\end{align}
	which finishes the proof of correctness.
\end{proof}

\section{Finiteness of the algorithm} \label{sec:fin}

In this section, we prove the finiteness of \Cref{alg1} by constructing a subset of fixed volume in each iteration. We show that the subsets are non-overlapping and also contained in a compact set, which implies a finite upper bound on the number of iterations of the algorithm. The result is based on the following lemma.

\begin{lemma} \label{lem:Bset}
	Suppose that \Cref{assmp:C_poly} holds. Let $v \notin \mathcal{P}$. Then, the following statements hold:
	\begin{enumerate}[(a)]
		\item It holds $\norm{z^v} = \of{\tilde{w}^v}^{\mathsf{T}}z^v$.
		\item It holds $\norm{\tilde{w}^v}_\ast=1$.
		\item If $\norm{z^v} \geq \epsilon$, then $\mathbb{B}_{\frac{\epsilon}{4}}(v) \cap \mathcal{H}(\tilde{w}^v, y^v,\epsilon) = \emptyset$, where $\mathcal{H}(\tilde{w}^v, y^v,\epsilon) \negthinspace=\negthinspace \{y \in \mathbb{R}^q \negthinspace\mid\negthinspace (\tilde{w}^v)^{\mathsf{T}}y \geq (\tilde{w}^v)^{\mathsf{T}} y^v - \frac{\epsilon}{2}\}$.
		\item It holds $\mathcal{H}(\tilde{w}^v, y^v)+ \mathbb{B}_{\frac{\epsilon}{2}}(0) \subseteq \mathcal{H}(\tilde{w}^v, y^v,\epsilon)$, where $\mathcal{H}(\tilde{w}^v, y^v)=\{y \in \mathbb{R}^q \mid (\tilde{w}^v)^{\mathsf{T}}y \geq (\tilde{w}^v)^{\mathsf{T}} y^v \}$.
	\end{enumerate}
\end{lemma}

\begin{proof}
	Using similar arguments as in the proof of \cite[Lemma 7.1]{ararat2021norm}, (a) and (c) follow. Moreover, (d) follows from \Cref{lem:ShiftedH}. To see (b), note that
	\[
	\norm{z^v}=\of{\tilde{w}^v}^{\mathsf{T}}z^v\leq \norm{\tilde{w}^v}_\ast\norm{z^v}\leq\norm{z^v} ,
	\]
	which follows by (a), H\"{o}lder's inequality, and the feasibility of $\tilde{w}^v$ for the formulation in \eqref{eq:pre_dualpr_eq}. Hence, all terms are equal. Since $z^v\neq 0$ by \Cref{lem:vlemma}(i), we get $\norm{\tilde{w}^v}_\ast=1$. 
\end{proof}

\begin{theorem}\label{thm:finiteness}
	Suppose that \Cref{assmp:C_poly} holds. Then, \Cref{alg1} terminates after a finite number of iterations.
\end{theorem}

\begin{proof}
	First, note that $\ext \bar{\mathcal{P}}^\out_k$ consists of finitely many vertices for each $k \geq 0$. Next, we show the existence of some $K \geq 0$ such that $\|z^v\| \leq \epsilon$ holds for every $v\in\ext \bar{\mathcal{P}}^\out_K$. Assume otherwise, that is, for every $k \geq$ 0, we have $\|z^{v^k}\| > \epsilon$, where $v^k$ is as in \eqref{eq:zmax}.
	
	By \cite[Lemma 6.3]{ararat2021norm}, the set $\bar{\mathcal{P}}^\out_0$ is compact. The ball $\mathbb{B}_{\frac{\epsilon}{2}}(0)$ is also compact. Hence, $\bar{\mathcal{P}}^\out_0 + \mathbb{B}_{\frac{\epsilon}{2}}(0)$ is compact by \cite[Lemma~5.3]{guide2006infinite}. For each $k\geq 0$, since $v^k \in \bar{\mathcal{P}}^\out_k$, we have
	\begin{equation} \label{eq:Bk_in_Pk}
		\mathbb{B}_{\frac{\epsilon}{4}}(v^k) \subseteq \{v^k\} + \mathbb{B}_{\frac{\epsilon}{2}}(0) \subseteq \bar{\mathcal{P}}^\out_k + \mathbb{B}_{\frac{\epsilon}{2}}(0) \subseteq \bar{\mathcal{P}}^\out_0 + \mathbb{B}_{\frac{\epsilon}{2}}(0).
	\end{equation} 
	Next, we show that $\mathbb{B}_{\frac{\epsilon}{4}}(v^i)\cap \mathbb{B}_{\frac{\epsilon}{4}}(v^j) = \emptyset$ for every $j > i \geq 0$. Clearly, $\bar{\mathcal{P}}^\out_j \subseteq \bar{\mathcal{P}}^\out_{i+1}$. By \Cref{lem:Bset}(c), we have $\mathbb{B}_{\frac{\epsilon}{4}}(v^i) \cap \mathcal{H}(\tilde{w}^{v^i},y^{v^i},\epsilon)= \emptyset.$ Moreover, we have
	\begin{align}\notag 
		\bar{\mathcal{P}}^\out_j + \mathbb{B}_{\frac{\epsilon}{2}}(0)  \subseteq \bar{\mathcal{P}}^\out_{i+1} + \mathbb{B}_{\frac{\epsilon}{2}}(0)= (\bar{\mathcal{P}}^\out_{i} \cap \mathcal{H}_{i}) + \mathbb{B}_{\frac{\epsilon}{2}}(0)\subseteq \mathcal{H}_{i} +\mathbb{B}_{\frac{\epsilon}{2}}(0),
	\end{align}
	where $\mathcal{H}_i$ is the supporting halfspace of $\mathcal{P}\cap S(\gamma)$ at $y^i$ as obtained in \Cref{prop:supp_halfspace}. Using \Cref{lem:Bset}(d) with the above inclusion, we get
	\begin{align}\notag 
		\bar{\mathcal{P}}^\out_j + \mathbb{B}_{\frac{\epsilon}{2}}(0)  \subseteq \mathcal{H}_{i} + \mathbb{B}_{\frac{\epsilon}{2}}(0)\subseteq \mathcal{H}(\tilde{w}^{v^i},y^{v^i},\epsilon).
	\end{align}
	This implies $\mathbb{B}_{\frac{\epsilon}{4}}(v^i) \cap (\bar{\mathcal{P}}^\out_{j} + \mathbb{B}_{\frac{\epsilon}{2}}(0) = \emptyset$. From \eqref{eq:Bk_in_Pk}, $\mathbb{B}_{\frac{\epsilon}{4}}(v^j) \subseteq \bar{\mathcal{P}}^\out_j + \mathbb{B}_{\frac{\epsilon}{2}}(0)$. Hence, $\mathbb{B}_{\frac{\epsilon}{4}}(v^i) \cap \mathbb{B}_{\frac{\epsilon}{4}}(v^j) = \emptyset$. Note that the sets $\mathbb{B}_{\frac{\epsilon}{4}}(v^k)$, $k\geq 0$, have the same volume which is strictly positive. With \eqref{eq:Bk_in_Pk}, these imply that the compact set $\bar{\mathcal{P}}^\out_0 + \mathbb{B}_{\frac{\epsilon}{2}}(0)$ contains an infinite number of disjoint subsets of identical and nonzero volume, a contradiction.
\end{proof}

\section{Convergence rate of the algorithm} \label{sec:conv}

In order to study the convergence rate of \Cref{alg1}, we aim to use the results of Kamenev \cite{kamenev1992class} and Lotov et al. \cite[Chapter 8]{lotov2013interactive}, which originally hold for convex compact bodies.

For the theoretical analysis in this section and in \Cref{sec:conv_l2}, we ignore the stopping criteria in \Cref{alg1} and assume that it runs indefinitely while updating the current outer approximation in each iteration. This can be done by ignoring lines 21 and 25 of the algorithm. 

\begin{remark}\label{cor:limit}
	Under \Cref{assmp:C_poly}, for the above modification of \Cref{alg1}, it holds
	\begin{align}\notag 
		\lim_{k\to \infty} \delta^H(\bar{\mathcal{P}}^\inn_k, \mathcal{P}) = \lim_{k\to \infty} \delta^H(\bar{\mathcal{P}}^\out_k + C, \mathcal{P}) = 0,
	\end{align} 
	where, for each $k\geq 0$, $\bar{\mathcal{P}}^\inn_{k} \coloneqq \conv \Gamma(\bar{\mathcal{X}}_{k}) + C$ and the sets $\bar{\mathcal{X}}_{k}, \bar{\mathcal{P}}^\out_k$ are as described in \Cref{alg1}. The proof follows the same lines as \cite[Corollary 7.4]{ararat2021norm} combined with \Cref{thm:finiteness}.
\end{remark}

In the following, we recall a definition from \cite[Chapter 8]{lotov2013interactive}. Note that the notations used in the original definition are replaced with the ones in this paper to remain consistent. To use the convergence result, \cite[Theorem 8.6]{lotov2013interactive}, our algorithm has to produce the type of sequence of outer approximating polytopes provided in \Cref{def:LotovH} below. 

\begin{definition} \cite[Definition~8.3]{lotov2013interactive}\label{def:LotovH} Consider a nonempty convex compact set $A \subseteq\R^q$ and a sequence $(A_k)_{k\geq 0}$ of polytopes in $\R^q$. Assume that $A_0 = \bigcap_{i=1}^I \mathcal{H}(\omega^i,A)$, where $I \in \mathbb{N}$ and $\omega^i \in\R^q\setminus\{0\}$ for $i\in\{1,\ldots, I\}$. We say that $(A_k)_{k\geq 0}$ is \emph{generated by a cutting method} if, for every $k\geq 0$, it holds $A_k \supseteq A$ and there exists a supporting halfspace $H_k \subseteq \R^q$ of $A$ such that $A_{k+1} = A_k \cap H_k$. In this case, $(A_k)_{k\geq 0}$ is called an \emph{$H(r,A)$-sequence of cutting} for a given constant $r > 0$ if, for every $k \geq 0$, it holds
	\begin{align}\notag
		\delta^H(A_k, A_{k+1}) \geq r \delta^H(A_k, A).
	\end{align}
\end{definition}

Note that, in each iteration $k\geq 0$ of the algorithm, we choose a farthest vertex, denoted as $v^k$, to the upper image of the current outer approximation and generate a halfspace $\mathcal{H}_k$, given by \eqref{eq:Hk}, to update the current outer approximation. Using this update structure, we show in \Cref{prop:H_seq_eq}  that 
the Hausdorff distance between any two consecutive outer approximations, $\bar{\mathcal{P}}^\out_k$ and $\bar{\mathcal{P}}^\out_{k+1}$, is equal to the Hausdorff distance between the former outer approximation $\bar{\mathcal{P}}^\out_k$ and $\mathcal{P} \cap S(\gamma)$.

The next lemma is a simple observation that will be used in \Cref{prop:H_seq_eq}. 

\begin{lemma} \label{lem:Bsetzv}
	Suppose that \Cref{assmp:C_poly} holds. Let $v\in S(\gamma)\setminus \mathcal{P}$. Then, $\Int \mathbb{B}_{\norm{z^v}}(v) \cap \mathcal{H}(\tilde{w}^v, y^v)  = \emptyset$, where the halfspace $\mathcal{H}(\tilde{w}^v, y^v)$ is defined by \Cref{prop:supp_halfspace}. 
\end{lemma}

\begin{proof}
	First, note that {$v \in S(\gamma)\setminus\mathcal{P}$} implies $\norm{z^v} > 0$ and $\tilde{w}^v\neq 0$ by \Cref{lem:vlemma}(a). Hence, $\Int\mathbb{B}_{\norm{z^{v}}}(v)\neq \emptyset$ and $\mathcal{H}(\tilde{w}^v, y^v)$ is well-defined. Using similar arguments as in the proof of \cite[Lemma 7.1]{ararat2021norm}, the assertion of the lemma follows.
\end{proof}

\begin{theorem}\label{prop:H_seq_eq}
	Suppose that \Cref{assmp:C_poly} holds and consider the sequence of outer approximating polytopes $(\bar{\mathcal{P}}^\out_k)_{k \geq 0}$ produced by \Cref{alg1}. Then, for every $k\geq0$, we have
	\[
	\delta^H(\bar{\mathcal{P}}^\out_k, \bar{\mathcal{P}}^\out_{k+1}) = \delta^H(\bar{\mathcal{P}}^\out_k, \mathcal{P} \cap S(\gamma)).
	\]
\end{theorem}

\begin{proof}
	Let $k\geq 0$. By construction, we have $\bar{\mathcal{P}}^\out_{k+1} \subseteq \bar{\mathcal{P}}^\out_k$. By \cite[Lemma 5.3]{ararat2021norm}, we get
	\begin{align}\notag 
		\delta^H(\bar{\mathcal{P}}^\out_k, \bar{\mathcal{P}}^\out_{k+1})
		= \max_{v \in \ext \bar{\mathcal{P}}^\out_k} \ d(y, \bar{\mathcal{P}}^\out_{k+1}).
	\end{align}
	First, we consider the case $\ext \bar{\mathcal{P}}^\out_k \subseteq \mathcal{P}$. Since $\bar{\mathcal{P}}^\out_k$ is a polytope, we have $\bar{\mathcal{P}}^\out_k \subseteq \mathcal{P}$ by the convexity of $\mathcal{P}$. Then, from $\bar{\mathcal{P}}^\out_k \subseteq \bar{\mathcal{P}}^\out_{0} = \mathcal{P}^\out_0 \cap S(\gamma) \subseteq S(\gamma)$, we get $\bar{\mathcal{P}}^\out_k \subseteq \mathcal{P} \cap S(\gamma)$. The reverse inclusion is trivial by construction so that $\bar{\mathcal{P}}^\out_k = \mathcal{P} \cap S(\gamma)$. Then, $\bar{\mathcal{P}}^\out_{k+1} = \bar{\mathcal{P}}^\out_k$ and $\delta^H(\bar{\mathcal{P}}^\out_k, \bar{\mathcal{P}}^\out_{k+1}) = 0$. Hence, the equality is trivial. 
	
	Next, let us assume that there exists some $v\in \ext\bar{\mathcal{P}}^\out_k$ such that $v \notin \mathcal{P}$. Without loss of generality, we may assume that $v=v^k$; recall that $v^k$ is a farthest vertex to $\mathcal{P} \cap S$; see \eqref{eq:zmax}. Then, $ \Int \mathbb{B}_{\|z^{v^k}\|}(v^k) \cap \mathcal{H}_k = \emptyset$ by \Cref{lem:Bsetzv}. Hence, $v^k \notin \bar{\mathcal{P}}^\out_k \cap \mathcal{H}_k = \bar{\mathcal{P}}^\out_{k+1}$. This implies $\delta^H(\bar{\mathcal{P}}^\out_k, \bar{\mathcal{P}}^\out_{k+1}) \geq \|z^{v^k}\| > 0$, {where the second inequality is due to} $v^k \notin \mathcal{P}$; see \Cref{lem:vlemma}. By \cite[Lemma 5.3]{ararat2021norm}, $\delta^H(\bar{\mathcal{P}}^\out_k, \mathcal{P} \cap S(\gamma))$ is attained at a vertex of $\bar{\mathcal{P}}^\out_k$ {so that}
	\begin{align}\notag 
		\delta^H(\bar{\mathcal{P}}^\out_k, \mathcal{P} \cap S(\gamma))
		= \max_{v \in \ext \bar{\mathcal{P}}^\out_k} \ d(v, \mathcal{P} \cap S(\gamma)) 
		= \max_{v \in \ext \bar{\mathcal{P}}^\out_k} \ \norm{z^v} 
		= \|z^{v^{k}}\|.
	\end{align}
	Let $\bar{y} \in \mathcal{H}_k$ be arbitrary. By the definition of $\mathcal{H}_k$ in \eqref{eq:Hk} and \Cref{lem:Bset}(a), we have $(\tilde{w}^{v^{k}})^{\mathsf{T}}\bar{y} \geq (\tilde{w}^{v^{k}})^{\mathsf{T}}(v^{k}+z^{v^{k}}) = (\tilde{w}^{v^{k}})^{\mathsf{T}} v^{k} + \|z^{v^{k}}\|$, that is,
	\begin{align}\label{eq:gennorm_3}
		(\tilde{w}^{v^{k}})^{\mathsf{T}}(\bar{y}-v^{k}) \geq \|z^{v^{k}}\|.
	\end{align} 
	
	On the other hand, by H\"older's inequality and \eqref{eq:pre_dualpr_eq}, we have
	\begin{align}\label{eq:gennorm_4}
		\abs{(\tilde{w}^{v^{k}})^{\mathsf{T}}(\bar{y}-v^{k})} \leq \|(\tilde{w}^{v^{k}})\|_{\ast}\|\bar{y}-v^{k}\| \leq \|\bar{y}-v^{k}\|.
	\end{align}
	Then, using \eqref{eq:gennorm_3} and \eqref{eq:gennorm_4}, we obtain $\|\bar{y}-v^{k}\| \geq \abs{(\tilde{w}^{v^{k}})^{\mathsf{T}}(\bar{y}-v^{k})} \geq \|z^{v^{k}}\|$. We get
	\begin{align}\notag 
		\delta^H(\bar{\mathcal{P}}^\out_k, \mathcal{P} \cap S(\gamma))
		= \|z^{v^{k}}\| & \leq \min_{y \in \mathcal{H}_k} \|y - v^{k}\| \\ \notag
		& \leq \min_{y \in \bar{\mathcal{P}}^\out_k \cap \mathcal{H}_k} \|y - {v^{k}}\| \\ \notag
		& = d(v^{k}, \bar{\mathcal{P}}^\out_k \cap \mathcal{H}_k) \\ \notag
		& \leq \max_{v \in \ext\bar{\mathcal{P}}^\out_k} \ d(v, \bar{\mathcal{P}}^\out_k \cap \mathcal{H}_k) \\ \notag
		& = \max_{v \in \ext\bar{\mathcal{P}}^\out_k} \ d(v, \bar{\mathcal{P}}^\out_{k+1}) 
		= \delta^H(\bar{\mathcal{P}}^\out_k, \bar{\mathcal{P}}^\out_{k+1}),
	\end{align}
	where the last equality follows from the fact that $\delta^H(\bar{\mathcal{P}}^\out_k, \bar{\mathcal{P}}^\out_{k+1})$ is attained at a vertex of $\bar{\mathcal{P}}^\out_k$; see \cite[Lemma 5.3]{ararat2021norm}. Hence, $\delta^H(\bar{\mathcal{P}}^\out_k, \bar{\mathcal{P}}^\out_{k+1}) \geq \delta^H(\bar{\mathcal{P}}^\out_k, \mathcal{P} \cap S(\gamma))$. Next, since $\bar{\mathcal{P}}^{\out}_k \supseteq \bar{\mathcal{P}}^{\out}_{k+1} \supseteq \mathcal{P} \cap S(\gamma)$, we have $\delta^H(\bar{\mathcal{P}}^\out_k, \bar{\mathcal{P}}^\out_{k+1}) \leq \delta^H(\bar{\mathcal{P}}^\out_k, \mathcal{P} \cap S(\gamma))$. Hence, the equality {in the theorem} follows.
\end{proof}

From now on, we use the following notation:
\begin{align}\label{eq:defA}
	\mathcal{A} := \mathcal{P} \cap S(\gamma); \qquad 
	\mathcal{A}_k := \bar{\mathcal{P}}^\out_k, \quad
	k \geq 0.
\end{align}
The next corollary verifies that these sets form an instance of \Cref{def:LotovH}.

\begin{corollary}\label{cor:H_seq}
	Suppose that \Cref{assmp:C_poly} holds. Then, the sequence of outer approximating polytopes $(\mathcal{A}_k)_{k \geq 0}$ is an $H(1,\mathcal{A})$-sequence of cutting, that is, $\delta^H(\mathcal{A}_k, \mathcal{A}_{k+1}) \geq \delta^H(\mathcal{A}_k, \mathcal{A})$ holds for every $k\geq 0$.
\end{corollary}

\begin{proof}
	Let $k\geq 0$. Clearly, $\mathcal{A}$ and $\mathcal{A}_k$ satisfy the required conditions for the sets in \Cref{def:LotovH}, since $\mathcal{A} \subseteq \mathcal{A}_k$ is compact by \cite[Lemma 6.3]{ararat2021norm}; $\mathcal{A}^0$ is defined as intersection of supporting halfspaces of $\mathcal{A}$, see lines 3 and 6 of \Cref{alg1}; and $\mathcal{A}_k$ has the required property, see line 22 of \Cref{alg1} and \Cref{prop:supp_halfspace}.
	Moreover, by the definitions of $\mathcal{A}$ and $\mathcal{A}_k$, and \Cref{prop:H_seq_eq}, we have $\delta^H(\mathcal{A}_k, \mathcal{A}_{k+1}) = \delta^H(\mathcal{A}_k, \mathcal{A})$.
\end{proof}

We restate an important theorem from \cite{lotov2013interactive} that will be useful in proving the convergence of $(\mathcal{A}_k)_{k \geq 0}$.

\begin{theorem}\cite[Theorem~8.5]{lotov2013interactive} \label{thm:lotovlimit} \label{thm:Lotovlimit}
	Let $A\subseteq \R^q$ be a nonempty convex compact set and $r>0$. For an $H(r,A)$-sequence $(A_k)_{k \geq 0}$, we have
	\begin{align}\notag
		\lim_{k\to \infty} \delta^H(A_k, A) = 0.
	\end{align}
\end{theorem}

Next, we relate the Hausdorff distance between each approximation produced by \Cref{alg1} and the upper image to the one between their compact versions in \eqref{eq:defA}.

\begin{lemma} \label{lem:PkAk}
	Suppose that \Cref{assmp:C_poly} holds. Then, for every $k\geq 0$, we have
	\begin{align}\notag 
		\delta^H(\bar{\mathcal{P}}_k^\out + C, \mathcal{P}) \leq \delta^H(\mathcal{A}_k, \mathcal{A}). 
	\end{align}
\end{lemma}

\begin{proof}
	Let $k\geq 0$. By \cite[Lemma 5.3]{ararat2021norm}, $\delta^H(\mathcal{A}_k, \mathcal{A})$ is attained at a vertex of $\mathcal{A}_k$. Hence,
	\begin{align}\notag 
		\delta^H(\mathcal{A}_k, \mathcal{A}) = \max_{v \in \ext \mathcal{A}_k} \ d(v, \mathcal{A}).
	\end{align}
	By \Cref{lem:PconeP}, we have $\ext (\mathcal{A}_k+C)\subseteq \ext \mathcal{A}_k$. Hence,
	\begin{align}\notag 
		\delta^H(\mathcal{A}_k, \mathcal{A}) & = \max_{v \in \ext \mathcal{A}_k} \ d(v, \mathcal{A}) \\ \notag
		& \geq \max_{v \in \ext \mathcal{A}_k} \ d(v, \mathcal{A}+C) \\ \notag
		& \geq \max_{v \in\ext (\mathcal{A}_k+C)} \ d(v, \mathcal{A}+C) \\ \notag
		& = \delta^H(\mathcal{A}_k + C, \mathcal{A}+C) \\ \notag
		& = \delta^H(\bar{\mathcal{P}}_k^\out + C, (\mathcal{P} \cap S(\gamma)) + C),
	\end{align}
	where the penultimate equality is a consequence of \Cref{rem:Scontains}, \Cref{rem:Pequal}, and \cite[Lemma 5.3]{ararat2021norm}. Moreover, by \Cref{rem:Pequal}, we have $\mathcal{P} = (\mathcal{P} \cap S(\gamma)) + C$. Hence, $\delta^H(\mathcal{A}_k, \mathcal{A}) \geq \delta^H(\bar{\mathcal{P}}_k^\out + C, \mathcal{P})$.
\end{proof}

\begin{remark}
	By \Cref{cor:H_seq}, $(\mathcal{A}_k)_{k \geq 0}$ is an $H(1,\mathcal{A})$-sequence of polytopes. Then, \Cref{thm:lotovlimit} and \Cref{lem:PkAk} together imply the following result, which is already discussed in \Cref{cor:limit}:
	\begin{align}\notag
		\lim_{k\to \infty} \delta^H(\bar{\mathcal{P}}_k^\out + C, \mathcal{P}) = \lim_{k\to \infty} \delta^H(\mathcal{A}_k, \mathcal{A}) = 0.
	\end{align}	
\end{remark}

We restate an important theorem from \cite{lotov2013interactive} on convergence rates. Its proof can be found in \cite{kamenev1992class}.

\begin{theorem}\cite[Theorem~8.6]{lotov2013interactive} \label{thm:LotovH}
	Let $A\subseteq \R^q$ be a nonempty convex compact set and $r>0$. Let $(A_k)_{k\geq 0}$ be an $H(r,A)$-sequence of cutting. Then, for every $\epsilon \in (0,1)$, there exists $N\geq 0$ such that
	such that
	\begin{align}\notag
		\delta^H(A_k, A) \leq (1 + \epsilon) \lambda(r,A) k^{\frac{1}{1-q}},
	\end{align}
	holds for every $k \geq N$. Here, $\lambda(r,A)$ depends on the topological properties of $A$ and can be found in \cite[Theorem 2]{kamenev1992class}.
\end{theorem}

We conclude this section with an application of the previous theorem in our setting.

\begin{corollary}
	Suppose that \Cref{assmp:C_poly} holds. Then, the approximation error obtained through the iterations of \Cref{alg1} decreases by the order of $\mathcal{O}(k^{1/(1-q)})$.
\end{corollary}	

\begin{proof}
	By \Cref{cor:H_seq}, the sequence of outer approximating polytopes $(\mathcal{A}_k)_{k \geq 0}$ is an $H(1,\mathcal{A})$-sequence of cutting. Then, by \Cref{lem:PkAk} and \Cref{thm:LotovH}, we get
	\begin{align}\notag 
		\delta^H(\bar{\mathcal{P}}_k^\out + C, \mathcal{P}) \leq \delta^H(\mathcal{A}_k, \mathcal{A}) \leq \lambda(r,\mathcal{A}) k^{\frac{1}{1-q}}
	\end{align}
	for all sufficiently large $k$. Hence, the result follows.
\end{proof}

\section{Improved convergence rate under Euclidean norm}\label{sec:conv_l2}

In this section, we find an improved {estimate of} convergence rate {for} \Cref{alg1} when the $\ell_2$-norm is used in the scalarizations. {The statement of \Cref{thm:Lotov} below is similar to those of \cite[Theorem~8.14]{lotov2013interactive} and \cite[Corollary~1]{kamenev2002conjugate}. However, our statement is valid for the outer approximating polytopes produced by \Cref{alg1} instead of the polytopes described in these references. We need the following assumption for the improved convergence rate.} 

\begin{assumption}\label{assump:norm}
	{$\norm{\cdot}$ is the $\ell_2$-norm, that is, $\norm{z}=\sqrt{z^{\mathsf{T}}z}$ for every $z\in\R^q$.}
\end{assumption}

The next theorem is the main result of this section.

\begin{theorem}\label{thm:Lotov} 
	Suppose that Assumptions~\ref{assmp:C_poly}, \ref{assump:norm} hold and consider the sequence of outer approximating polytopes $(\bar{\mathcal{P}}^\out_k)_{k\geq 0}$ produced by \Cref{alg1}.	
	Then, for every $\epsilon\in(0,1)$, there exists $N\geq 0$
	such that
	\begin{align}\notag
		\delta^H(\bar{\mathcal{P}}^\out_{k-1}, \mathcal{P} \cap S(\gamma)) \leq (1 + \epsilon) \bar{\lambda}(\mathcal{P} \cap S(\gamma)) k^{\frac{2}{1-q}}
	\end{align}
	holds for every $k\geq N$. Here, $\bar{\lambda}(\mathcal{P}\cap S(\gamma))$ is a constant that depends on the topological properties of $\mathcal{P}\cap S(\gamma)$ and can be seen in \Cref{notation:constants}. 
	
	\begin{notation}\label{notation:constants} Let $A \subseteq \R^q$ be a nonempty convex compact set. 
		\begin{enumerate}[(a)]
			\item $R(A)$ is the radius of the smallest ball circumscribed around $A$, $r(A)$ is the radius of the largest ball inscribed in $A$ and $\omega(A)\coloneqq \frac{R(A)}{r(A)}$ is the asphericity of $A$. 
			\item The hypervolume of the unit ball $\mathbb{B}_1(0) \subseteq \R^q$ is denoted by $\pi_q$ and $\bar{\lambda}(A)\coloneqq 16R(A)\of{\frac{q \pi_q}{\pi_{q-1}}}^{\frac{2}{q-1}}$. 
		\end{enumerate} 
	\end{notation}
\end{theorem}

We present five lemmas to prepare for the proof of \Cref{thm:Lotov}. While two of these lemmas are directly paraphrased from \cite{lotov2013interactive}, we provide full proofs for the remaining three.

\begin{lemma} \cite[Lemma 8.15]{lotov2013interactive} \label{lem:lemma1}
	Suppose that \Cref{assump:norm} holds. Let $y, y^{\prime} \in \R^q$ and $w, w^{\prime} \in \mathbb{S}^{q-1}$ be such that $w^{\mathsf{T}}w^\prime > 0$, $y^{\prime} \in \mathcal{H}(w, y)$, and $y \in \mathcal{H}(w^{\prime}, y^{\prime})$. Then, for every $\eta>0$, we have
	\begin{align}\notag
		d(y, \bd \mathcal{H}(w^{\prime}, y^{\prime})) \leq \frac{\norm{(y - \eta w) - (y^{\prime} - \eta w^{\prime})}^2}{\eta}.
	\end{align}
\end{lemma}

The next lemma is an analogue of \cite[Lemma~8.17]{lotov2013interactive} in our setting.

\begin{lemma} \label{lem:lemma2}
	Suppose that Assumptions~\ref{assmp:C_poly}, \ref{assump:norm} hold.	Let $v\in S(\gamma)\setminus \mathcal{P}$ and set $y\coloneqq y^v$, $w\coloneqq \tilde{w}^v$. Let $w^\prime \in \mathbb{S}^{q-1}\cap C^+$, $y^\prime\in \mathcal{P}\cap S(\gamma)$ be such that $(w^\prime)^{\mathsf{T}}y^\prime=\inf_{z\in \mathcal{P}\cap S(\gamma)} (w^\prime)^{\mathsf{T}}z$. Assume that $v \in \mathcal{H}(w^\prime, \mathcal{P} \cap S(\gamma))$. Then, for every $\eta>0$, we have
	\begin{align}\notag
		\norm{(y - \eta w) - (y^\prime - \eta w^\prime)} \geq \min \{\eta \sqrt{2}, \sqrt{\eta h} -  h\},
	\end{align}
	where $h \coloneqq w^\mathsf{T} (y -v)$.
\end{lemma}

\begin{proof}
	Note that $\norm{w} = 1$ by \Cref{lem:Bset}(b). First, let us suppose that $w^\mathsf{T} w' \leq 0$. By \Cref{prop:supp_halfspace}, we have $w^\mathsf{T} ( y' - y) \geq 0$; by the definition of $y^\prime$, we have  $(w')^\mathsf{T} (y - y') \geq 0$. Hence, 
	\begin{align}\notag
		\lVert (y - \eta w) - (y' - \eta w')\rVert^2
		&= \norm{y - y'}^2 + \eta^2 \norm{w' - w}^2 + 2 \eta w^\mathsf{T} (y' - y) + 2 \eta (w')^\mathsf{T} (y - y')\\ \notag 
		&\geq \eta^2 \norm{w' - w}^2 \\ \notag
		& = \eta^2 \norm{w'}^2 + \eta^2 \norm{w}^2 - 2 \eta^2 w^\mathsf{T} w' \\ \notag
		& \geq \eta^2 \norm{w'}^2 + \eta^2 \norm{w}^2  = 2\eta^2
	\end{align}
	so that $\norm{(y - \eta w) - (y' - \eta w')} \geq \eta \sqrt{2}$.
	
	Next, suppose that $w^\mathsf{T} w' > 0$. By triangle inequality, we have
	\begin{equation}\label{eq:tri}
		\norm{(y - \eta w) - (y' - \eta w')}
		\geq \norm{(v - \eta w) - (y' - \eta w')} - \norm{(v - \eta w) - (y - \eta w)},
	\end{equation}
	and we control each term on the right separately. For the first term in \eqref{eq:tri}, note that $y' \in \mathcal{P}\cap S(\gamma)\subseteq \mathcal{H}(w, \mathcal{P} \cap S(\gamma))= \mathcal{H}(w, y)$. Moreover, by \Cref{lem:Bset}(a), we have $\mathcal{H}(w,y)\subseteq \mathcal{H}(w,v)$. Hence, $y^\prime \in \mathcal{H}(w,v)$. On the other hand, since $v \in S(\gamma)\setminus \mathcal{P}$, \Cref{lem:Bsetzv} ensures that $v \notin \mathcal{H}(w,y)$. By our assumption, we have
	\begin{align}\notag
		v \in \mathcal{H}(w', \mathcal{P} \cap S(\gamma)) = \mathcal{H}(w', y').
	\end{align}
	Then, by \Cref{lem:lemma1}, we get
	\[
	\norm{(v - \eta w) - (y' - \eta w')}^2
	\geq \eta d(y', \bd \mathcal{H}(w, v)).
	\]
	Since $\bd \mathcal{H}(w,v)=\{z\in\R^q\mid w^\mathsf{T}z= w^\mathsf{T}v\}$ and $w\in \mathbb{S}^{q-1}$, by elementary geometry, we have
	\[
	d(y', \bd \mathcal{H}(w, v)) = \abs{(w)^\mathsf{T}(y' - v)} \geq w^\mathsf{T}(y'-y) + w^\mathsf{T}y - w^\mathsf{T}v \geq w^\mathsf{T}y - w^\mathsf{T}v = h
	\]
	so that $\norm{(v - \eta w) - (y' - \eta w')}\geq \sqrt{\eta h}$. For the second term in \eqref{eq:tri}, by \Cref{lem:Bset} (a), we have
	\[
	\norm{(v - \eta w) - (y - \eta w)}
	= \norm{v-y}  = \norm{\tilde{w}^v} w^{\mathsf{T}}(y-v)= \norm{\tilde{w}^v} h \leq h 
	\]
	Hence, $\norm{(y - \eta w) - (y' - \eta w')}\geq \sqrt{\eta h} - h$ by \eqref{eq:tri}.
\end{proof}

Let us fix some $\eta > 0$. Similar to the construction \cite[Chapter 8, p.~253]{lotov2013interactive}, we define a new sequence $(\mathcal{Z}_k)_{k\geq 0}$ as follows:
\begin{enumerate}[(i)]
	\item We set $\mathcal{U}_0 \coloneqq \{(w^j,\Gamma(x^j)) \mid j \in\{1,\ldots,J\}\} \cup \{ (-\bar{w},\bar{y})\}$, where $\{(x^j,w^j)\mid j\in\{1,\ldots,J\}\}$ is as in \eqref{eq:P_0}, $\bar{w}$ is defined in \eqref{eq:gamma}, and $\bar{y}\in \mathcal{P}\cap S(\gamma)$ is such that $\bar{w}^\mathsf{T}\bar{y} = \gamma$. The existence of $\bar{y}$ is guaranteed since the supremum is attained (hence finite) by the compactness of $\mathcal{X}$ and the continuity of $x \mapsto \bar{w}^{\mathsf{T}}\Gamma(x)$. Let us define
	\[
	\mathcal{Z}_0 \coloneqq \{y-\eta w\mid (w,y)\in\mathcal{U}_0\}=\{\Gamma(x^1) - \eta w^1, \dots , \Gamma(x^J) - \eta w^J, \bar{y} + \eta \bar{w} \},
	\]
	\item For each $k\geq 0$, if $v^k \notin \mathcal{P}$, then we define 
	\begin{align}\label{eq:Zk}
		\mathcal{U}_{k+1}&\coloneqq \{(w^{v^k},y^{v^k})\}\cup\mathcal{U}_k,\\ \notag \mathcal{Z}_{k+1} &\coloneqq \{y-\eta w\mid (w,y)\in\mathcal{U}_{k+1}\}= \{y^{v^k} - \eta \tilde{w}^{v^k}
		\} \cup \mathcal{Z}_k.
	\end{align}
	\item If $v^{\bar{k}} \in \mathcal{P}$ for some $\bar{k}\geq 0$, then  by the vertex selection rule in \Cref{alg1} (line 20), the current outer approximation is the same as $\mathcal{P}\cap S(\gamma)$. In this case, we set $\mathcal{Z}_{k} = \mathcal{Z}_{\bar{k}}$ for all $k \geq \bar{k}$.
	(Note that this case is realized only if $\mathcal{P}\cap S(\gamma)$ is a polyhedral set.) 
\end{enumerate}

\begin{remark}\label{rem:Zkbd}
	For each $k\geq 0$, we have $\mathcal{Z}_k \subseteq \bd (\mathcal{P} \cap S(\gamma) +  \mathbb{B}_\eta(0))$. Let $k\geq0$ and $(w,y)\in \mathcal{U}_k$. Note that we have %
	\begin{align}\notag
		\mathcal{P} \cap S(\gamma) +  \mathbb{B}_\eta(0) \subseteq \mathcal{H}(w,y)+  \mathbb{B}_\eta(0).
	\end{align}
	Moreover, as $\norm{w}=1$, $\mathcal{H}(w,y)+  \mathbb{B}_\eta(0) \subseteq \mathcal{H}(w,y,2\eta)$ holds by \Cref{lem:ShiftedH}. Thus, we obtain $\mathcal{P} \cap S(\gamma) + \mathbb{B}_\eta(0)
	\subseteq \mathcal{H}(w,y,2\eta).$ Clearly, $y - \eta w \in \mathcal{P} \cap S(\gamma) + \mathbb{B}_\eta(0)$. Moreover, we have $y - \eta w \in \bd \mathcal{H}(w,y,2\eta)$ since $		w^\mathsf{T} (y - \eta w)= w^\mathsf{T} y - \eta$.
	These imply that $y - \eta w \in \bd(\mathcal{P}\cap S(\gamma)+\mathbb{B}_\eta(0))$. 
\end{remark}

The next lemma shows that the cardinality of the sequence $(\mathcal{Z}_k)_{k\geq 0}$ grows linearly.

\begin{lemma}\label{lem:cardZk}
	Suppose that Assumptions~\ref{assmp:C_poly}, \ref{assump:norm} hold. Suppose also that $\mathcal{P}\cap S(\gamma)$ is a non-polyhedral set. Then, $\abs{\mathcal{Z}_k}=J+1+k$ for each $k\geq 0$.
\end{lemma}

\begin{proof}
	Note that since $\mathcal{P}\cap S(\gamma)$ is a non-polyhedral set, $\delta^H(\bar{\mathcal{P}}^\out_k, \mathcal{P}\cap S(\gamma)) > 0$ holds for all $k\geq 0$. This implies that for all $k\geq 0$, (1) $v^k \notin \mathcal{P}$ by the vertex selection rule of \Cref{alg1} and (2) $\bar{\mathcal{P}}^\out_{k+1} \neq \bar{\mathcal{P}}^\out_k$ since, by \Cref{cor:H_seq}, we have
	\begin{align}\notag
		\delta^H(\bar{\mathcal{P}}^\out_{k+1}, \bar{\mathcal{P}}^\out_k) \geq  \delta^H(\bar{\mathcal{P}}^\out_k, \mathcal{P}\cap S(\gamma))>0.
	\end{align}
	Now, by construction, $\abs{\mathcal{Z}_0}= J+1$ and $\abs{\mathcal{Z}_k}\leq J+1+k$ for each $k\geq 0$. To get a contradiction, let $\bar{k}\geq 0$ be the smallest index at which we have $\abs{\mathcal{Z}_{\bar{k}+1}}<J+1+\bar{k}+1$. 
	Then, there exists $(w,y) \in \mathcal{U}_{\bar{k}}$ such that 
	\begin{align}\label{eq:ykwk}
		a \coloneqq y^{v^{\bar{k}}}-\eta \tilde{w}^{v^{\bar{k}}} = y -\eta w.
	\end{align}
	We claim that $y^{v^{\bar{k}}}=y$. To prove this, let us note that, for an arbitrary $b\in\bd \mathcal{H}_{\bar{k}}$, we can calculate the distance from $a$ to the hyperplane $\bd\mathcal{H}_{\bar{k}}$ as
	\[
	d(a,\mathcal{H}_{\bar{k}}) = \frac{\abs{({\tilde{w}}^{v^{\bar{k}}})^{\mathsf{T}} (a - b)}}{\|{\tilde{w}}^{v^{\bar{k}}}\|}=\abs{({\tilde{w}}^{v^{\bar{k}}})^{\mathsf{T}} (a - b)},
	\]
	where we use \Cref{lem:Bset}(b). In particular, taking $b= y^{v^{\bar{k}}}\in \bd \mathcal{H}_{\bar{k}}$ gives
	\[
	d(a,\mathcal{H}_{\bar{k}}) = \abs{(\tilde{w}^{v^{\bar{k}}})^\mathsf{T} (a - y^{v^{\bar{k}}})} = \eta \abs{(\tilde{w}^{v^{\bar{k}}})^\mathsf{T} \tilde{w}^{v^{\bar{k}}}} = \eta \|\tilde{w}^{v^{\bar{k}}}\|^2 = \eta.
	\]
	The above distance is attained at $y^{v^{\bar{k}}}$ since $\|a-y^{v^{\bar{k}}}\|=\eta\|\tilde{w}^{v^{\bar{k}}}\|^2=\eta$. Moreover, due to the strict convexity of the $\ell_2$-norm, we have $\|a-y^\prime\|>\|a-y^{v^{\bar{k}}}\|=\eta$ for every $y^\prime\in\mathcal{H}_{\bar{k}}$ with $y^\prime\neq y^{v^{\bar{k}}}$. In particular, having $y^{v^{\bar{k}}}\neq y$ would yield $y\in\mathcal{P}\cap S(\gamma)\subseteq \mathcal{H}_{\bar{k}}$ so that $\norm{a-y}>\eta$. However, we also have $\norm{a-y}=\eta \norm{w}=\eta$. Hence, we must have $y^{v^{\bar{k}}}=y$. Moreover, by \eqref{eq:ykwk}, we also have $\tilde{w}^{v^{\bar{k}}}=w$ and $\mathcal{H}_{\bar{k}}=\mathcal{H}(w,y)$. Note that by the structure of \Cref{alg1}, we have 
	\[
	\bar{\mathcal{P}}^\out_{\bar{k}} = \bigcap_{(w,y) \in \mathcal{U}_{\bar{k}}} \mathcal{H}(w,y). 
	\]
	Then, $\bar{\mathcal{P}}^\out_{{\bar{k}}+1}  = \bar{\mathcal{P}}^\out_{{\bar{k}}} \cap \mathcal{H}_{\bar{k}}= \bar{\mathcal{P}}^\out_{{\bar{k}}} $, which is a contradiction. 
	Hence, $\abs{\mathcal{Z}_k}=J+1+k$ for each $k\geq 0$.
\end{proof}

Before proceeding further, we recall the following definition, which is critical in proving the convergence rate of the algorithm in \Cref{thm:Lotov}.

\begin{definition}\label{defn:packing}\cite[Chapter 8]{lotov2013interactive}
	A set $Z\subseteq \R^q$ is called \emph{the base of an $\epsilon$-packing} if $\norm{y-z}\geq 2\epsilon$ for every $y,z\in Z$.
\end{definition}

The following lemma is originally proved for an $H_1$-sequence of outer approximating polytopes, see \cite[Definition 8.4 and Lemma 8.19]{lotov2013interactive}. Here, we show that the lemma still holds for the sequence of outer approximating polytopes generated by \Cref{alg1} which, in general, may not be an $H_1$-sequence.

\begin{lemma} \label{lem:lemma3}
	Suppose that Assumptions~\ref{assmp:C_poly}, \ref{assump:norm} hold. Suppose that $0<\epsilon<1$. Then, there exists $N\geq 0$ such that, for each $k\geq N$, the set $\mathcal{Z}_k$ is the base of an $\varepsilon_k^N(\epsilon)$-packing, where
	\begin{align}\label{eq:epsilonkN}
		\varepsilon_k^N( \epsilon) \coloneqq \frac12 \min \cb{\eta \sqrt{2}, 2 \epsilon_N, (1 - \epsilon) \sqrt{\eta \delta^H(\bar{\mathcal{P}}^\out_{k-1}, \mathcal{P} \cap S(\gamma))}}
	\end{align}
	and $\epsilon_N \coloneqq \frac12\min \cb{\norm{y-z} \mid y,z \in \mathcal{Z}_N, y \neq z}$.
\end{lemma}

\begin{proof}
	Let us fix $N\geq 0$ to be determined later. We argue by induction on $k\geq N$. By definitions of $\epsilon_N$ and $\varepsilon_N^N( \epsilon)$, we have $\norm{y-z} \geq 2\epsilon_N \geq 2\varepsilon_N^N( \epsilon)$ for every $y,z \in \mathcal{Z}_N$. Hence, the assertion is trivial for $k=N$.
	
	Assume that $\mathcal{Z}_{k-1}$ is the base of an $\varepsilon_{k-1}^N( \epsilon)$-packing for some $k\geq N+1$, that is,
	\begin{align}\notag
		\norm{y-z} \geq 2 \varepsilon_{k-1}^N( \epsilon)
	\end{align}
	for every $y,z\in\mathcal{Z}_{k-1}$. Moreover, since $\delta^H(\bar{\mathcal{P}}^\out_{k-2}, \mathcal{P} \cap S(\gamma)) \geq \delta^H(\bar{\mathcal{P}}^\out_{k-1}, \mathcal{P} \cap S(\gamma))$, by \eqref{eq:epsilonkN}, we have
	\begin{align}\notag
		\varepsilon_{k-1}^N( \epsilon) \geq \varepsilon_k^N( \epsilon).
	\end{align}
	Hence, $\mathcal{Z}_{k-1}$ is also the base of an $\varepsilon_k^N( \epsilon)$-packing.	
	
	For convenience, let us define $b^{k-1} \coloneqq y^{v^{k-1}} - \eta \tilde{w}^{v^{k-1}} = v^{k-1} + z^{v^{k-1}} - \eta \tilde{w}^{v^{k-1}}$. Since $\mathcal{Z}_k=\{b^{k-1}\}\cup\mathcal{Z}_{k-1}$, to show that $\mathcal{Z}_k$ is the base of an $\varepsilon_k^N( \epsilon)$-packing, it is enough to verify that
	\begin{align}\notag
		\norm{b^{k-1}- b} \geq 2\varepsilon_k^N( \epsilon)
	\end{align}
	for every $b \in \mathcal{Z}_{k-1}$. Let us fix $b\in\mathcal{Z}_{k-1}$. We may write $b = y' - \eta w'$, where $(w^\prime,y^\prime)\in \mathcal{U}_{k-1}$.  
	
	We know that $v^{k-1} \in S(\gamma)\setminus \mathcal{P}$ and $v^{k-1} \in \mathcal{H}(w^\prime,y^\prime) = \mathcal{H}(w^\prime,\mathcal{P} \cap S(\gamma))$. Also, $y' \in \bd \mathcal{H}(w', y')$ and $y^\prime\in \mathcal{P} \cap S(\gamma)$ by definition. Then, by \Cref{lem:lemma2}, we have
	\begin{align}\label{eq:normb}
		\|b^{k-1} - b\| = \|(y^{v^{k-1}} - \eta \tilde{w}^{v^{k-1}}) - (y' - \eta w')\| \geq \min \cb{\eta \sqrt{2}, \ \sqrt{\eta h^{k-1}} - h^{k-1}},
	\end{align}
	where $h^{k-1} \coloneqq (\tilde{w}^{v^{k-1}})^\mathsf{T} (y^{v^{k-1}} - v^{k-1})$. Note that
	\begin{align}\notag
		h^{k-1} &= (\tilde{w}^{v^{k-1}})^\mathsf{T} (y^{v^{k-1}} - v^{k-1}) =(\tilde{w}^{v^{k-1}})^\mathsf{T} z^{v^{k-1}}
		=\|z^{v^{k-1}}\|
		=\delta^H(\bar{\mathcal{P}}^\out_{k-1}, \mathcal{P} \cap S(\gamma)),
	\end{align}
	
	We choose $N$ so that $\delta^H(\bar{\mathcal{P}}^\out_N, \mathcal{P} \cap S(\gamma)) \leq \eta \epsilon^2.$ In particular, for each $k\geq N+1$, we have $\delta^H(\bar{\mathcal{P}}^\out_{k-1}, \mathcal{P} \cap S(\gamma)) \leq \eta \epsilon^2$, which implies
	\begin{align}\notag
		\sqrt{\frac{\eta \epsilon^2}{\delta^H(\bar{\mathcal{P}}^\out_{k-1}, \mathcal{P} \cap S(\gamma))}} \geq 1.  
	\end{align}
	Then, 
	\begin{align}\notag
		\sqrt{\eta h^{k-1}} - &  h^{k-1}
		= \sqrt{\eta \delta^H(\bar{\mathcal{P}}^\out_{k-1}, \mathcal{P} \cap S(\gamma))}  -  \delta^H(\bar{\mathcal{P}}^\out_{k-1}, \mathcal{P} \cap S(\gamma)) \\ \notag
		& \geq \sqrt{\eta \delta^H(\bar{\mathcal{P}}^\out_{k-1}, \mathcal{P} \cap S(\gamma))}  - \sqrt{\frac{\eta \epsilon^2}{\delta^H(\bar{\mathcal{P}}^\out_{k-1}, \mathcal{P} \cap S(\gamma))}} \ \delta^H(\bar{\mathcal{P}}^\out_{k-1}, \mathcal{P} \cap S(\gamma)) \\ \notag
		& = (1 - \epsilon) \sqrt{\eta \delta^H(\bar{\mathcal{P}}^\out_{k-1}, \mathcal{P} \cap S(\gamma))}.
	\end{align}
	Hence, from \eqref{eq:normb}, we have
	\begin{align}\notag
		\norm{b^{k-1} - b} &\geq \min \cb{\eta \sqrt{2}, \ \sqrt{\eta h^{k-1}} -  h^{k-1}} \\ \notag
		& \geq \min \cb{\eta \sqrt{2}, \ (1 - \epsilon) \sqrt{\eta \delta^H(\bar{\mathcal{P}}^\out_{k-1}, \mathcal{P} \cap S(\gamma))}} \\ \notag
		& \geq \min \cb{\eta \sqrt{2}, 2 \epsilon_N, (1 - \epsilon) \sqrt{\eta \delta^H(\bar{\mathcal{P}}^\out_{k-1}, \mathcal{P} \cap S(\gamma))}}  = 2\varepsilon_k^N( \epsilon).\notag 
	\end{align}
\end{proof}

\begin{lemma} \cite[Lemma 8.20]{lotov2013interactive} \label{lem:lemma4}
	Let $A\subseteq\R^q$ be a nonempty convex compact set. Suppose that $0<\epsilon<R(A)$ and let $Z\subseteq \bd A$ be the base of an $\epsilon$-packing. Then,
	\[
	\abs{Z}\leq N_1 (\epsilon, R(A)) \coloneqq \frac{q \pi_q}{\pi_{q-1}} \of{1 + \frac{R(A)^2}{\epsilon^2}}^{\frac{q-1}{2}}.
	\]
\end{lemma}

Using the above lemmas, we are ready to prove \Cref{thm:Lotov}.

\begin{proof}[\textbf{Proof of \Cref{thm:Lotov}}]
	If $\mathcal{P} \cap S(\gamma)$ is a polytope, then there exists $K'\in\N$ such that $\delta^H(\bar{\mathcal{P}}^\out_k, \mathcal{P} \cap S(\gamma))=0$ for every $k \geq K'$, and the assertion of the theorem holds trivially. For the rest of the proof, we assume that $\mathcal{P} \cap S(\gamma)$ is a non-polyhedral set.
	
	Let $\epsilon^\prime\coloneqq \frac12(1-\frac{1}{\sqrt{1+\epsilon}})\in (0,1)$ and $\eta\coloneqq R(\mathcal{P} \cap S(\gamma))>0$. By \Cref{lem:lemma3}, there exists $N^\prime\in \N$ such that, for every $k \geq N'$, the set $\mathcal{Z}_k$ is the base of an $\varepsilon_k^{N'}( \epsilon')$-packing, where
	\begin{align}\notag
		& \varepsilon_k^{N'}( \epsilon') = \min \cb{  \frac{\eta\sqrt{2}}{2},  \epsilon_{N'}, \tau_k},\\ \notag &\epsilon_{N^\prime} = \frac12 \min \cb{\norm{y-z} \mid y,z \in \mathcal{Z}_{N^\prime}, y \neq z},\\ \notag 
		&\tau_k \coloneqq \frac{(1 - \epsilon') \sqrt{\eta \delta^H(\bar{\mathcal{P}}^\out_{k-1}, \mathcal{P} \cap S(\gamma))}}{2}.\notag 
	\end{align}
	Moreover, by \Cref{cor:H_seq}, $(\bar{\mathcal{P}}^\out_k)_{k\geq 0}$ is an $H(1,\mathcal{P} \cap S(\gamma))$-sequence of cutting. Hence, by \Cref{thm:Lotovlimit}, there exists $N'' \geq N'$ such that
	\begin{equation}\label{eq:deltabeta}
		\delta^H(\bar{\mathcal{P}}^\out_{N''-1}, \mathcal{P} \cap S(\gamma)) \leq \frac{4}{\eta} \of{\frac{\varepsilon_{N'}^{N'}( \epsilon')}{1 - \epsilon'}}^2,
	\end{equation}
	where $\varepsilon_{N'}^{N'}( \epsilon')>0$ holds as $\mathcal{P}\cap S(\gamma)$ is not a polyhedral set. Hence, we get
	\begin{align}\notag
		\tau_{N^{\prime\prime}}=\frac{(1 - \epsilon') \sqrt{\eta \delta^H(\bar{\mathcal{P}}^\out_{N''-1}, \mathcal{P} \cap S(\gamma))}}{2} \leq \varepsilon_{N'}^{N'}( \epsilon') \leq  \min \cb{\frac{\eta \sqrt{2}}{2},  \epsilon_{N'}}.
	\end{align}
	Observe that for every $k \geq N''$, we have $\tau_{k}\leq \tau_{N^{\prime\prime}}$; hence
	\begin{align}\notag
		\varepsilon_k^{N'}( \epsilon') = \min \cb{  \frac{\eta\sqrt{2}}{2},  \epsilon_{N'}, \tau_k}= \tau_{k};
	\end{align}
	in particular, $\mathcal{Z}_k$ is the base of a $\tau_k$-packing. Similar to \eqref{eq:deltabeta}, using  \Cref{thm:Lotovlimit}, we may find $ N^{\prime\prime\prime}\in\N$ such that, for every $k \geq N^{\prime\prime\prime}$,
	\[
	\delta^H(\bar{\mathcal{P}}^\out_{k-1}, \mathcal{P} \cap S(\gamma)) \leq \frac{4 (R(\mathcal{P} \cap S(\gamma)+ \mathbb{B}_{\eta}(0)))^2}{(1-\epsilon^\prime)^2\eta},
	\]
	which is equivalent to 
	\[
	\tau_k \leq R(\mathcal{P} \cap S(\gamma) + \mathbb{B}_{\eta}(0)).
	\]
	By \Cref{rem:Zkbd}, $\mathcal{Z}_k \subseteq \bd (\mathcal{P} \cap S(\gamma) +  \mathbb{B}_{\eta}(0))$. Then, from \Cref{lem:lemma4}, for every $k \geq \max \{N'', N^{\prime\prime\prime}\}$, we get
	\begin{align}\notag
		\abs{\mathcal{Z}_k} \leq N_1 (\tau_k, R(\mathcal{P} \cap S(\gamma)+  \mathbb{B}_{\eta}(0))).
	\end{align}
	Moreover, by definition of $R(\cdot)$ and by the choice of $\eta$, we have 
	\begin{align}\notag
		N_1 (\tau_k, R(\mathcal{P} \cap S(\gamma) + \mathbb{B}_{\eta}(0)))  &= N_1 (\tau_k, R(\mathcal{P} \cap S(\gamma)) + \eta) \\ \notag
		&= \frac{q \pi_q}{\pi_{q-1}} \cb{1 + \of{\frac{2R(\mathcal{P} \cap S(\gamma))}{\tau_k}}^2}^{\frac{q-1}{2}}. \notag
	\end{align}
	Since $\mathcal{P} \cap S(\gamma)$ is a non-polyhedral set, by \Cref{lem:cardZk}, for every $k\geq 0$, we have
	\begin{align}\notag
		\abs{\mathcal{Z}_k} = J + 1 + k.
	\end{align}
	Hence, for every $k\geq \max\{N^{\prime\prime}, N^{\prime\prime\prime}\}$, we have
	\begin{align}\notag
		k  \leq \abs{\mathcal{Z}_k} \leq N_1 (\tau_k, 2R(\mathcal{P} \cap S(\gamma)))=\frac{q \pi_q}{\pi_{q-1}} \cb{1 + \frac{16R(\mathcal{P} \cap S(\gamma))}{(1 - \epsilon')^2 \delta^H(\bar{\mathcal{P}}^\out_{k-1}, \mathcal{P} \cap S(\gamma))}}^{\frac{q-1}{2}},
	\end{align}
	which implies that
	\begin{equation}\label{eq:deltamiddleterm}
		\delta^{H}(\bar{\mathcal{P}}^{\out}_{k-1},\mathcal{P}\cap S(\gamma))\leq \frac{16R(\mathcal{P}\cap S(\gamma))}{(1-\epsilon^\prime)^2\sqb{\of{\frac{q\pi_q}{k\pi_{q-1}}}^{\frac{2}{1-q}}-1}}.
	\end{equation}

	Next, we will show that 
	\begin{equation}\label{eq:g0}
		\frac{16R(\mathcal{P}\cap S)}{(1-\epsilon^\prime)^2\sqb{\of{\frac{q\pi_q}{k\pi_{q-1}}}^{\frac{2}{1-q}}-1}}\leq (1+\epsilon)16R(\mathcal{P}\cap S(\gamma))\of{\frac{q\pi_q}{k\pi_{q-1}}}^{\frac{2}{q-1}}
	\end{equation} 
	holds for sufficiently large $k$. Note that \eqref{eq:g0} holds if \begin{equation}\label{eq:g}
		g(k)\coloneqq \frac{1}{1-\of{\frac{q\pi_q}{k\pi_{q-1}}}^{\frac{2}{q-1}}}-1\leq (1+\epsilon)(1-\epsilon^\prime)^2-1
	\end{equation}
	holds. Here, $g$ is a decreasing function on $\N$ with $\lim_{k\rightarrow\infty}g(k)=0$. Moreover, by the choice of $\epsilon^\prime$, we have $(1+\epsilon)(1-\epsilon^\prime)^2-1 > 0$. Then, there exists $N^{\prime\prime\prime\prime}\in \N$ such that \eqref{eq:g}, hence \eqref{eq:g0}, hold for all $k \geq N^{\prime\prime\prime\prime}$.
	
	From \eqref{eq:deltamiddleterm} and \eqref{eq:g0}, we obtain that 
	\[
	\delta^{H}(\bar{\mathcal{P}}^{\out}_{k-1},\mathcal{P}\cap S(\gamma))\leq (1+\epsilon)16R(\mathcal{P}\cap S(\gamma))\of{\frac{q\pi_q}{k\pi_{q-1}}}^{\frac{2}{q-1}},
	\]
	holds for $k\geq \max\{N^{\prime\prime},N^{\prime\prime\prime},N^{\prime\prime\prime\prime}\}$, as desired.
\end{proof}
By \Cref{thm:Lotov}, we prove that the approximation error obtained through the iterations of \Cref{alg1}, when the Euclidean norm is used in the scalarization, decreases by the order of $\mathcal{O}(k^{2/1-q})$.

\section{Examples and computational results}
\label{sec:experiments}

In this section, we consider some numerical examples to observe the performance of \Cref{alg1}, which is implemented using MATLAB R2018a along with CVX, a package to solve convex programs \cite{cvx, gb08}, and \emph{bensolve tools} \cite{lohne2017vector} to solve the scalarization and vertex enumeration problems in each iteration, respectively. The tests are performed using a 3.6 GHz Intel Core i7 computer with a 64 GB RAM. 

We consider three examples:
\begin{enumerate}
	\item Let $q\in \{2,3,4\}$, $e=(1,\ldots,1)^{\mathsf{T}}\in \R^q$ and  $C=\R_+^q$.
	\begin{align} 
		\text{minimize~~} & x \text{~~with respect to} \leq_{C} \notag \\
		\text{subject to~~} &	\norm{x - e}_2 \leq 1, \:\: x\in\R^q. \notag
	\end{align}
	\item Let $a^1=(1,1)^\mathsf{T},a^2=(2,3)^\mathsf{T}, a^3=(4,2)^\mathsf{T}$. 
	\begin{align}
		\text{minimize~~} & (\norm{x-a^1}_2^2,\norm{x-a^2}_2^2,\norm{x-a^3}_2^2)^\mathsf{T}
		\text{~~with respect to} \leq_{\R_+^3} \notag  \\
		\text{subject to~~} & x_1 + 2x_2 \leq 10, \:\: 0 \leq x_1 \leq 10, \:\: 0 \leq x_2 \leq 4, \:\: x\in\R^2. \notag
	\end{align}
	\item Let $b^1=(0,10,-120), b^2=(80, -448, 80), b^3=(-448,80,80)$.
	\begin{align}
		\text{minimize~~} & ( \norm{x}_2^2+b^1x,  \norm{x}_2^2+b^2x, \norm{x}_2^2+b^3x)^\mathsf{T}
		\text{~~with respect to} \leq_{\R_+^3} \notag  \\
		\text{subject to~~} & \norm{x}_2^2 \leq 100, \:\: 0 \leq x_i \leq 10 \text{~for~} i\in{\{1,2,3\}}, \:\: x\in\R^3. \notag
	\end{align}
\end{enumerate}

Example 1 is a standard illustrative example with a linear objective function, see \cite{ehrgott2011approximation,lohne2014primal}. In Example 2, the objective function is nonlinear while the constraints are linear; in Example 3, nonlinear terms appear both in the objective function and constraints \cite[Examples 5.8, 5.10]{ehrgott2011approximation}, \cite{miettinen2006experiments}. These examples have also been used recently in \cite{ararat2021norm} to test the performances of similar CVOP algorithms.   

We run \Cref{alg1} for these examples where we use the Euclidean norm within the scalarizations. To observe the convergence behavior, we measure the actual Hausdorff distance between the outer approximation and upper image at each iteration. In Figures \ref{fig:1} and \ref{fig:2}, we plot $\log \delta^H (\bar{\mathcal{P}}^\out_{k},\mathcal{P})$ versus $\log k$, we compare the linear regression of this graph with the graph of $\log (c k^{\frac{2}{1-q}})$, where $c \in \R_+$ is selected such that the slopes of the lines can be observed easily.\footnote{The slopes of the regression lines are indicated in the figures.} Except for Example 1 with $q=2$, the estimated rate of the decrease in the actual Hausdorff distance is better than the theoretical bound. Note that in Example 1 the image of the feasible region is the unit ball centered at $e \in \R^q$. This implies that the vertices of the outer approximations obtained through the iterations of the algorithm are distributed symmetrically with respect to the upper image, and many vertices have the same Euclidean distance to the upper image. Moreover, the number of vertices increase exponentially through the iterations. Hence, the actual Hausdorff distance remains constant for many iterations, which explains the step-function like behavior observed in Figure 1, especially for $q=2$.

\begin{figure}\label{fig:1}
	\caption{Example 1 with $q = 2, \epsilon = 10^{-5}$ (top left), $q = 3, \epsilon = 0.01$ (top right) and $q = 4, \epsilon = 0.0496$ (bottom).}
	\scalebox{0.48}{\begin{tikzpicture}
			\begin{axis}[
				xmin = 0, xmax = 6,
				ymin = -14, ymax = 2,
				xtick distance = 1,
				ytick distance = 2,
				grid = both,
				minor tick num = 1,
				major grid style = {lightgray!50},
				minor grid style = {lightgray!10},
				width = \textwidth,
				height = 0.75\textwidth,
				xlabel = {log $k$},
				ylabel = {log $\delta^H$},
				legend cell align = {left},
				legend pos = north east
				]
				
				\addplot[
				ultra thick,
				red,
				] file[skip first] {2d.dat};
				
				\addplot[
				domain = 0:5.5,
				thick,
				blue,
				] {-2*x + ln(1.5)};
				
				\addplot [very thick, black] table [
				x=x,
				y={create col/linear regression={y=y}}]
				{2d.dat}
				coordinate [pos=0.25] (A)
				coordinate [pos=0.4] (B)
				;
				\xdef\slope{\pgfplotstableregressiona}
				\draw (A) -| (B)
				node [pos=0.75,anchor=west]
				{\pgfmathprintnumber{\slope}};

				\legend{
					log $\delta^H$,
					log (1.5$k^{-2}$),
					linear regression		
				}
				
			\end{axis}
			
		\end{tikzpicture}
	} \scalebox{0.48}{\begin{tikzpicture}
			\begin{axis}[
				xmin = 0, xmax = 5,
				ymin = -6, ymax = 2,
				xtick distance = 1,
				ytick distance = 1,
				grid = both,
				minor tick num = 1,
				major grid style = {lightgray!50},
				minor grid style = {lightgray!10},
				width = \textwidth,
				height = 0.75\textwidth,
				xlabel = {log $k$},
				ylabel = {log $\delta^H$},
				legend cell align = {left},
				legend pos = north east
				]
				
				\addplot[
				ultra thick,
				red,
				] file[skip first] {3d.dat};
				
				\addplot[
				domain = 0:4.5,
				thick,
				blue,
				] {-x + ln(2.5)};
				
				\addplot [very thick, black] table [
				x=x,
				y={create col/linear regression={y=y}}]
				{3d.dat}
				coordinate [pos=0.25] (A)
				coordinate [pos=0.4] (B)
				;
				\xdef\slope{\pgfplotstableregressiona}
				\draw (A) -| (B)
				node [pos=0.75,anchor=west]
				{\pgfmathprintnumber{\slope}};

				\legend{
					log $\delta^H$,
					log (2.5$k^{-1}$),
					linear regression		
				}
				
			\end{axis}
			
		\end{tikzpicture}
	} \\ \centering \scalebox{0.48}{\begin{tikzpicture}
			\begin{axis}[
				xmin = 0, xmax = 5,
				ymin = -3.5, ymax = 2,
				xtick distance = 0.5,
				ytick distance = 0.5,
				grid = both,
				minor tick num = 1,
				major grid style = {lightgray!50},
				minor grid style = {lightgray!10},
				width = \textwidth,
				height = 0.75\textwidth,
				xlabel = {log $k$},
				ylabel = {log $\delta^H$},
				legend cell align = {left},
				legend pos = north east
				]
				
				\addplot[
				ultra thick,
				red,
				] file[skip first] {4d.dat};
				
				
				\addplot[
				domain = 0:4.8,
				thick,
				blue,
				] {-2/3*x + ln(4)};
				
				\addplot [very thick, black] table [
				x=x,
				y={create col/linear regression={y=y}}]
				{4d.dat}
				coordinate [pos=0.25] (A)
				coordinate [pos=0.4] (B)
				;
				\xdef\slope{\pgfplotstableregressiona}
				\draw (A) -| (B)
				node [pos=0.75,anchor=west]
				{\pgfmathprintnumber{\slope}};

				\legend{
					log $\delta^H$,
					log (4$k^{-2/3}$),
					linear regression
				}
				
			\end{axis}
			
	\end{tikzpicture}}
\end{figure}
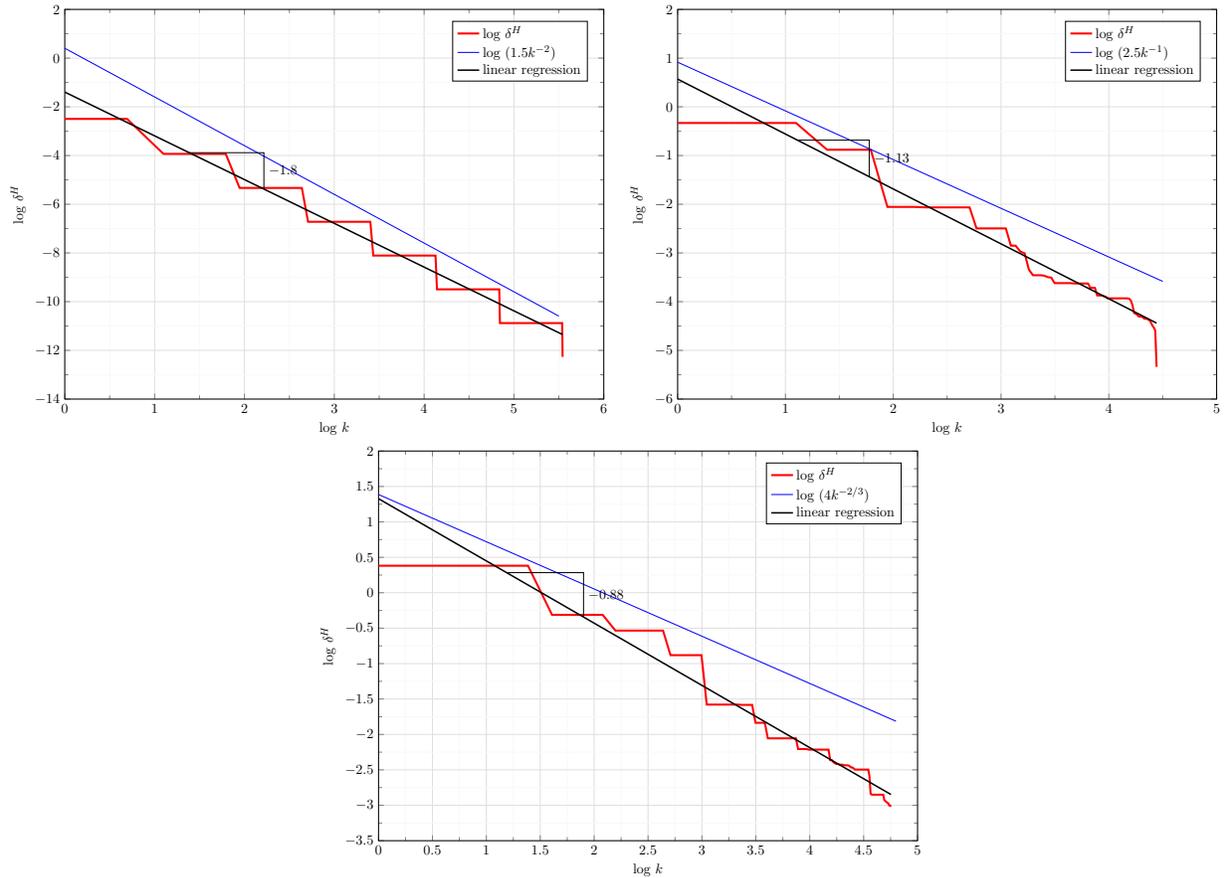

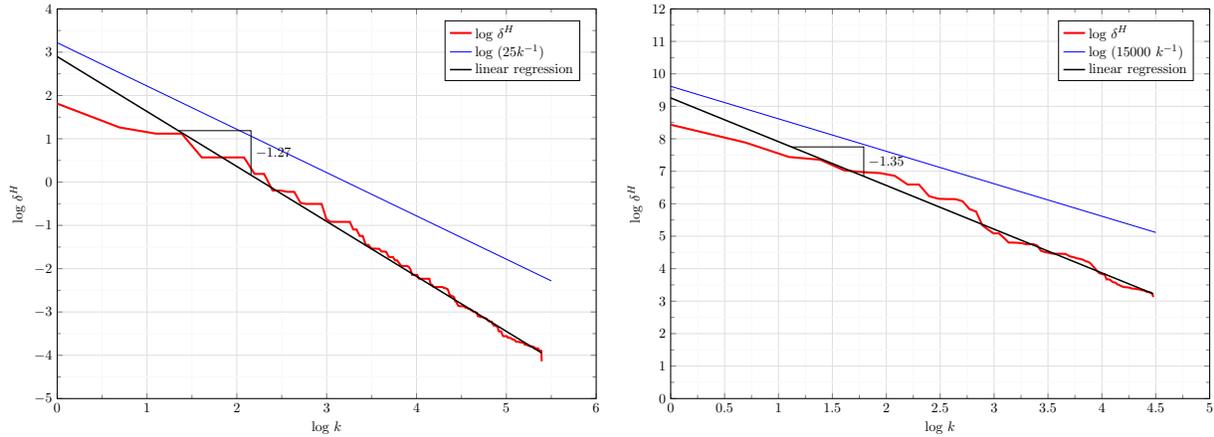
\begin{figure}[h]
	\label{fig:2}
	\caption{Example 2 with $\epsilon = 0.02$ (left) and Example 3 with $\epsilon=25$ (right).}
	\scalebox{0.48}{
		\begin{tikzpicture}
			\begin{axis}[
				xmin = 0, xmax = 6,
				ymin = -5, ymax = 4,
				xtick distance = 1,
				ytick distance = 1,
				grid = both,
				minor tick num = 1,
				major grid style = {lightgray!50},
				minor grid style = {lightgray!10},
				width = \textwidth,
				height = 0.75\textwidth,
				xlabel = {log $k$},
				ylabel = {log $\delta^H$},
				legend cell align = {left},
				legend pos = north east
				]
				
				\addplot[
				ultra thick,
				red,
				] file[skip first] {5_8.dat};
				
				\addplot[
				domain = 0:5.5,
				thick,
				blue,
				] {-x + ln(25)};
				
				\addplot [very thick, black] table [
				x=x,
				y={create col/linear regression={y=y}}]
				{5_8.dat}
				coordinate [pos=0.25] (A)
				coordinate [pos=0.4] (B)
				;
				\xdef\slope{\pgfplotstableregressiona}
				\draw (A) -| (B)
				node [pos=0.75,anchor=west]
				{\pgfmathprintnumber{\slope}};

				\legend{
					log $\delta^H$,
					log (25$k^{-1}$),
					linear regression		
				}
				
			\end{axis}
	\end{tikzpicture}
} \scalebox{0.48}{
		\begin{tikzpicture}
			\begin{axis}[
				xmin = 0, xmax = 5,
				ymin = 0, ymax = 12,
				xtick distance = 0.5,
				ytick distance = 1,
				grid = both,
				minor tick num = 1,
				major grid style = {lightgray!50},
				minor grid style = {lightgray!10},
				width = \textwidth,
				height = 0.75\textwidth,
				xlabel = {log $k$},
				ylabel = {log $\delta^H$},
				legend cell align = {left},
				legend pos = north east
				]
				
				\addplot[
				ultra thick,
				red,
				] file[skip first] {5_10.dat};
				
				\addplot[
				domain = 0:4.5,
				thick,
				blue,
				] {-x + ln(15000)};
				
				\addplot [very thick, black] table [
				x=x,
				y={create col/linear regression={y=y}}]
				{5_10.dat}
				coordinate [pos=0.25] (A)
				coordinate [pos=0.4] (B)
				;
				\xdef\slope{\pgfplotstableregressiona}
				\draw (A) -| (B)
				node [pos=0.75,anchor=west]
				{\pgfmathprintnumber{\slope}};

				\legend{
					log $\delta^H$,
					log (15000 $k^{-1}$),
					linear regression		
				}
				
			\end{axis}
			
	\end{tikzpicture}}
\end{figure}

\section{Conclusion}\label{sec:concl}

To the best of our knowledge, we study the convergence rate of a CVOP algorithm for the first time. The approximation error of the proposed algorithm decreases by the order of $\mathcal{O}(k^{{1}/{(1-q)}})$, where $q$ is the dimension of the objective space. Moreover, we consider the special case of using the Euclidean norm within the scalarizations separately and prove that the approximation error decreases by the order of $\mathcal{O}(k^{2/(1-q)})$. Since it is known that in multiobjective optimization ($C=\R^q_+$) this is the best possible convergence rate \cite{klamroth2007constrained, gruber1992aspects}, it cannot be improved in general. It is an open problem to prove the improved convergence rate for the proposed algorithm in the general setting when the underlying norm in the scalarizations is not necessarily the Euclidean norm.

\bibliographystyle{plain}
\bibliography{NormMinAlg_arxiv_new.bib}

\end{document}